\documentclass[12pt,a4paper]{amsart}
\usepackage[latin1]{inputenc}
\usepackage[english]{babel}
\usepackage{amsmath,amstext,amsthm,amsfonts,amssymb,amscd}
\usepackage[dvips]{graphicx}
\usepackage{hyperref}

\usepackage{cleveref}
\usepackage{color}

\theoremstyle{plain}
\newtheorem{theorem}{Theorem}[subsection]
\newtheorem{proposition}{Proposition}[subsection]
\newtheorem{definition}{Definition}[subsection]
\newtheorem{remark}{Remark}[subsection]

\newtheorem{example}{Example}[subsection]
\newtheorem{claim}{Claim}[subsection]
\newtheorem{lemma}{Lemma}[subsection]
\newtheorem{corollary}{Corollary}[subsection]
\newtheorem{theorema}{Theorem}

\newcommand{\topo}{\operatorname{top}}

\begin{document}

\title[Equilibrium States for Open Zooming Systems]{Equilibrium States for Open Zooming Systems}


\author[E. Santana]{Eduardo Santana}
\address{Eduardo Santana, Universidade Federal de Alagoas, 57200-000 Penedo, Brazil}
\email{jemsmath@gmail.com}

\dedicatory{Dedicated to the memory of Maur\'icio Peixoto}

\thanks{The author was partially supported by CNPq - Brazil,   under the project with reference 409198/2021-8.} 

\date{\today}

\maketitle

\begin{abstract} 
In this work, we construct Markov structures for zooming systems adapted to holes of a special type. Our construction is based on backward contractions provided by zooming times. These Markov structures may be used to code the open zooming systems.

In the context of open zooming systems, possibly with the presence of a critical/singular set, we prove the existence of finitely many ergodic zooming equilibrium states for zooming potentials whose induced potential is locally Hölder. For example, the zooming Hölder continuous. Among the zooming ones are the so-called \textit{hyperbolic potentials} and also what we call \textit{bounded distortion potentials}, having as a particular case the \textit{pseudo-geometric potentials} $\phi_{t} = -t \log J_{\mu}f $, where $J_{\mu}f$ is a Jacobian of the reference zooming measure. Moreover, for this last class of potentials, we show the existence of what we call pseudo-conformal measures. Afterwards, if there exist infintely many ergodic zooming measures, we prove uniqueness of equilibrium state with no requirement of transitivity. 

The technique consists in using an inducing scheme in a finite Markov structure with infinitely many symbols to code the dynamics to obtain an equilibrium state for the associated symbolic dynamics and then projecting it to obtain an equilibrium state for the original map. To obtain a pseudo-conformal measure, we "spread" the conformal one which exists for the symbolic dynamics. The uniqueness is obtained by showing that the equilibrium states are liftable to the same inducing scheme.

Finally, we show that the class of hyperbolic potentials is equivalent to the class of continuous zooming potentials. Moreover, we give a class of examples of hyperbolic potentials (including the null one). It implies the existence and uniqueness of equilibrium state (and measure of maximal entropy). Among the maps considered is the important class known as Viana maps.
\end{abstract}

\bigskip

\section{Introduction}
The theory of nonuniformly expanding maps has been widely studied and the expanding measures are related. They are particular cases of what we call zooming systems, related to zooming measures introduced in \cite{Pi1}. Here, we deal with open zooming systems.

Roughly speaking, a \textbf{\textit{zooming system}} is a map which extends the notion of non-uniform expansion,
where we have a type of expansion obtained in the presence of hyperbolic times. The zooming times extend this notion beyond exponential contractions.
In our context, a system is said to be \textbf{\textit{open}} when the phase space is not invariant. In other words, we begin with a closed map
$f:M \to M$ and consider a Borel set $H \subset M$, in order to study the orbits with respect to $H$ (called the \textbf{\textit{hole}} of the system):
if a point $x \in M$ is such that its orbit pass through $H$, we consider that the point $x$ escapes from the system. We observe that it reduces to a closed map when $H = \emptyset$.
In particular, we study the set of points that never pass through $H$, called the \textbf{\textit{survivor set}}. For zooming systems, we consider holes such that the survivor set contains the zooming set. Also, once a reference measure $m$ is fixed, we study the \textbf{\textit{escape rate}} defined by:
\[
\displaystyle \mathcal{E}(f,m,H) = - \lim_{n \to \infty}\frac{1}{n}\log m\Big{(}\cap_{j=0}^{n-1} f^{-j}(M \backslash H)\Big{)}.
\]
The study of open dynamics began with Pianigiani and Yorke in the late 1970s (see \cite{PY}). The abstract concept of an open system is important because it leads immediately to the notion of a conditionally invariant measure and escape rate along with a host of detailed questions about how mass escapes or fails to escape from the system under time evolution.

In order to study the escape rate, for example, Young towers are quite useful and the inducing schemes as well. For interval maps, inducing schemes are studied in \cite{DHL} and changes are made in \cite{BDM} to study inducing schemes for interval maps with holes. In the context of open dynamics, the inducing schemes need to be \textbf{\textit{adapted to the hole}}, as considered in \cite{DT}, for example (and referred to as \textbf{\textit{respecting the hole}}). It means that an element of the partition can only escape entirely.
It is an important property for constructing the Young towers with holes and what we construct here may be used for this. 

However, in this work we will not deal with escape rate and inducing schemes will play an important role in proving the existence of equilibrium states. We begin by proving the existence of inducing schemes adapted to holes of a special type. The context is the zooming systems and the hole is obtained from small balls, by erasing intersections with regular pre-images. Our construction follows the work of Pinheiro in \cite{Pi1} along the same lines, where the main ingredient is the zooming times.

The theory of equilibrium states in dynamical systems was first developed by Sinai, Ruelle and Bowen in the sixties and seventies. 
It was based on applications of techniques of Statistical Mechanics to smooth dynamics. The classical theory of equilibrium states is developed for continuous maps. Given a continuous map $f: M \to M$ on a compact metric space $M$ and a continuous potential $\phi : M \to \mathbb{R}$, an  \textbf{\textit{equilibrium state}} is an 
invariant measure that satisfies a variational principle, that is, a measure $\mu$ such that
\begin{equation}\label{equilibrium}
\displaystyle h_{\mu}(f) + \int \phi d\mu = \sup_{\eta \in \mathcal{M}_{f}(M)} \bigg{\{} h_{\eta}(f) + \int \phi d\eta \bigg{\}},
\end{equation}
where $\mathcal{M}_{f}(M)$ is the set of $f$-invariant probabilities on $M$ and $h_{\eta}(f)$ is the so-called metric entropy of $\eta$.

For measurable maps we define equilibrium states as follows.
Given a measurable map $f: M \to M$ on a compact metric space $M$ for which the set $\mathcal{M}_{f}(M)$ of invariant measures is non-empty (for example, the continuous maps) and a measurable potential $\phi : M \to \mathbb{R}$, we define the \textbf{\textit{pressure}} $P_{f}(\phi)$ as 
\begin{equation}\label{pressure}
\displaystyle P_{f}(\phi) : = \sup_{\eta \in \mathcal{M}_{f}(M)} \bigg{\{} h_{\eta}(f) + \int \phi d\eta \bigg{\}}
\end{equation}
and an \textbf{\textit{equilibrium state}} is an invariant measure that attains the supremum, and it is defined analogously to the case for continuous maps (see equation \ref{equilibrium}).

In the context of uniform hyperbolicity, which includes uniformly expanding maps, equilibrium states do exist and are unique if the potential is H\"older continuous
and the map is transitive. In addition, the theory for finite shifts was developed and used to achieve the results for smooth dynamics.

Beyond uniform hyperbolicity, the theory is still far from complete. It was studied by several authors, including Bruin, Keller, Demers, Li, Rivera-Letelier, Iommi and Todd 
 \cite{BDM, BK, BT,DT,IT1,IT2,LRL} for interval maps; Denker and Urbanski  \cite{DU} for rational maps; Leplaideur, Oliveira and Rios 
\cite{LOR} for partially hyperbolic horseshoes; Buzzi, Sarig and Yuri \cite{BS,Y}, for countable Markov shifts and for piecewise expanding maps in one and higher dimensions. 
For local diffeomorphisms with some kind of non-uniform expansion, there are results due to Oliveira \cite{O}; Arbieto, Matheus and Oliveira \cite{AMO};
Varandas and Viana \cite{VV}, all of whom proved the existence and uniqueness of equilibrium states for potentials with low oscillation. Also, for this type of map,
Ramos and Viana  \cite{RV}  proved it for the so-called  \textbf{\textit{hyperbolic potentials}}, which include these previous ones for the case of non-uniform expansion. The hyperbolicity of the potential is
characterized by the fact that the pressure emanates from the hyperbolic region. In most of these studies previously cited, the maps do not have the presence of critical sets. In \cite{BDM} and \cite{DT}, for example, the authors develop results for open interval maps with critical sets, but not for hyperbolic potentials and, recently, Alves, Oliveira and Santana proved the existence of equilibrium states for hyperbolic potentials, possibly with the presence of a critical set (see \cite{AOS}). We will see that the potentials considered in \cite{LRL}, which allow critical sets, are included in the potentials considered in \cite{AOS}. Here, we give an example of a class of hyperbolic potentials. It includes the null potential for Viana maps. We stress that one of the consequences in this paper is the existence and uniqueness of measures of maximal entropy for Viana maps and we observe that in \cite{ALP} the authors show, in particular, the existence of at most countably many ergodic measures of maximal entropy. In \cite{LOP} the authors prove the existence of at most one. In \cite{PV} Pinheiro-Varandas obtain existence and uniqueness by using another technique than us for what they call \textbf{\textit{expanding potentials}}. They are also considered in \cite{AOS}, include the hyperbolic ones and are included in our \textbf{\textit{zooming potentials}}. We prove that the class of hyperbolic potentials is equivalent to the class of continuous zooming potentials and that they include the null one.

For interval maps and the so-called \textbf{\textit{geometric potentials}} $\phi_{t} = -t\log \mid Df \mid$, equilibrium states were studied in  \cite{BT}, \cite{IT1}, \cite{IT2} and \cite{PS}. These results inspired ours, for measurable zooming systems on metric spaces, where the potential is $\phi_{t} = -t\log J_{\mu}f$ and $J_{\mu}f$ is a Jacobian of the reference measure and we call them \textbf{\textit{pseudo-geometric potentials}}, which are examples of measurable ones. We generalize this type of potential as what we call  \textbf{\textit{bounded distortion potentials}}. We also prove it for zooming potentials with finite pressure and locally Holder induced potential. All results are for open dynamics, which reduce to the corresponding closed dynamics if the hole is empty. We also mention the works \cite{DT2} and \cite{PU} for important comprehension of the theory of equilibrium states for open systems and Hölder and geometric potentials. 

Our strategy is also similar and follows exactly along the same lines of \cite{AOS}, since we do not use the analytical approach of the transfer operator in order to obtain conformal measures. We use results on countable Markov shifts by Sarig for the ``coded'' dynamics in inducing schemes constructed here, where a Markov structure is constructed. We prove that there exist finitely many ergodic equilibrium states
that are zooming measures. Moreover, by following ideas of Iommi and Todd in \cite{IT1} we obtain a \textbf{\textit{pseudo-conformal}} measure with no need for a transfer operator. The uniqueness of the equilibrium state is obtained afterwards. In our setting, if the survivor set is disjoint from the hole, then all the ergodic equilibrium states give full measure to the survivor set, which contains the zooming set.

Hence, we extend to the open context some results by many authors. For example, the recent work \cite{PV} by V. Pinheiro and P. Varandas, where they obtain existence and uniqueness of equilibrium states for closed zooming systems which are strongly topologically transitive and also with exponential contractions. Here, we use another technique and also obtain this for open systems, allowing nonexponential contractions, besides finiteness and also uniqueness of ergodic equilibrium states with no requirement of transitivity.

\bigskip

\textbf{Organization of the paper.} This paper is organized as follows. In Section \ref{Setup} we give the key definitions and state the main results. In particular, we define zooming systems and introduce zooming potentials. Section \ref{Preliminary} is devoted to basic definitions and results from the background theory. In Section \ref{Markov} we prove Theorem \ref{A} about the construction of a Markov structure for open systems with special holes. The section \ref{Lifted} contains some results from the theory of O. Sarig for the thermodynamic formalism for countable shifts. In Section \ref{Finiteness} we prove part of Theorem \ref{B}, about finiteness of equilibrium states. Section \ref{Uniqueness} establishes the uniqueness. In section \ref{Hyperbolic} we prove that the class of hyperbolic potentials is equivalent to the class of continuous zooming potentials and give a class of hyperbolic potentials, which includes the null one. In particular, we obtain uniqueness of the measure of maximal entropy for Viana maps. Section \ref{Pseudo} is devoted to the introduction of what we call pseudo-geometric potentials and the proof of the existence of equilibrium states for them. In section \ref{Conformal} we prove the existence of what we call pseudo-conformal measures for the pseudo-geometric potentials. Section \ref{Distortion} presents a generalization of the pseudo-geometric potentials which we call them bounded distortion potentials.
Finally, section \ref{Examples} contains some examples of maps for which we can apply our theoretical results. We observe that almost all of the examples are closed systems and we provide an example of open system at the end of the section.

\bigskip

\textbf{Acknowledgements:} The author is grateful to Rafael A. Bilbao for his contribution in revising
the manuscript, to Paulo Ribenboim, Krerley Oliveira, José F. Alves, Carlos Matheus, Vilton Pinheiro, Ricardo Bortolotti, Yuri Lima, Davi Lima, Carllos E. Holanda, Ermerson Araujo, Jamerson Bezerra, Mauricio Poletti and Wagner Ranter for fruitful conversations and to the anonymous referees for their valuable comments which contributed to correct and improve the manuscript.

\section{Setup and Main Results}\label{Setup}

In this section we give some definitions and state our main results. 

\subsection{Zooming Sets and Measures}

For differentiable dynamical systems, hyperbolic times are a powerful tool to obtain a type of expansion in the context of non-uniform expansion. As we can find in \cite{AOS}, it can be generalized for systems considered in a metric space, also with exponential contractions. The zooming times generalizes it beyond the exponential
context. Details can be seen in \cite{Pi1}.

Let $f : M \to M$ be a measurable map defined on a connected, compact, separable metric space $M$.

\begin{definition}
(Zooming contractions). A \textbf{\textit{zooming contraction}} is a sequence of functions $\alpha_{n}: [0,+\infty) \to [0,+\infty)$ such that

\begin{itemize}
 \item $\alpha_{n}(r) < r, \text{for all} \, \, n \in \mathbb{N}, \text{for all} \, \, r>0.$
 
 \item $\alpha_{n}(r)<\alpha_{n}(s), \,\, if \,\, 0<r<s, \text{for all} \, \, n \in \mathbb{N}$.
 
 \item $\alpha_{m} \circ \alpha_{n}(r) \leq \alpha_{m+n}(r), \text{for all} \, \, r>0, \text{for all} \, \, m,n \in \mathbb{N}$.
  
 \item $\displaystyle \sup_{r \in (0,1)} \sum_{n=1}^{\infty}\alpha_{n}(r) < \infty$.
\end{itemize}

\end{definition}

As defined in \cite{PV}, we call the contraction $(\alpha_{n})_{n}$ \textbf{\textit{exponential}} if $\alpha_{n}(r) = e^{-\lambda n} r$ for some $\lambda > 0$ and \textbf{\textit{Lipschitz}} if $\alpha_{n}(r) = a_{n} r$ with $0 \leq a_{n} < 1, a_{m}a_{n} \leq a_{m+n}$ and $\sum_{n=1}^{\infty} a_{n} < \infty$. In particular, every exponential contraction is Lipschitz. We can also have the example with $a_{n} = (n+b)^{-a}, a > 1, b>0$.

\begin{definition}
(Zooming times). Let $(\alpha_{n})_{n}$ be a zooming contraction and $\delta>0$. We say that $n \in \mathbb{N}$ is an $(\alpha,\delta)$\textbf{\textit{-zooming time}} for $p \in M$ if there exists 
a neighbourhood $V_{n}(p)$ of $p$ such that

\begin{itemize}
\item $f^{n}$ sends $\overline{V_{n}(p)}$ homeomorphically onto $\overline{B_{\delta}(f^{n}(p))}$;

\item $d(f^{j}(x),f^{j}(y)) \leq \alpha_{n - j}(d(f^{n}(x),f^{n}(y)))$ for every $x,y \in V_{n}(p)$ and every $0 \leq j < n$.
\end{itemize}

We call $B_{\delta}(f^{n}(p))$ a \textbf{\textit{zooming ball}} and $V_{n}(p)$ a  \textbf{\textit{zooming pre-ball}}.
\end{definition}

We denote by $Z_{n}(\alpha,\delta,f)$ the set of points in $M$ for which $n$ is an $(\alpha, \delta)$- zooming time.

\begin{definition}
(Zooming measure) A $f$-non-singular finite measure $\mu$ defined on the Borel sets of M is called a \textbf{\textit{weak zooming measure}} if $\mu$ almost every point has 
infinitely many $(\alpha, \delta)$-zooming times. A weak zooming measure is called a \textbf{\textit{zooming measure}} if
\end{definition}
\begin{equation}\label{frequency}
\displaystyle \limsup_{n \to \infty} \frac{1}{n} \{1 \leq j \leq n \mid x \in Z_{j}(\alpha,\delta,f)\} > 0,
\end{equation}
$\mu$ almost every $x \in M$.

\begin{definition}
(Zooming set) A forward invariant set $\Lambda \subset M$ (that is, $T(\Lambda) \subset \Lambda$) is called a \textbf{\textit{zooming set}} if the above inequality \ref{frequency} holds for every $x \in \Lambda$. 
\end{definition}
\begin{remark}\label{fixed}
We stress that measure $\mu$ is not necessarily ergodic and the definition of the zooming set $\Lambda$ depends on the contraction $\alpha$ and the parameter $\delta$. Moreover, every forward invariant subset $\Lambda_{0} \subset \Lambda$ ($f(\Lambda_{0}) \subset \Lambda_{0}$) is also a zooming set if $\Lambda_{0}$ is full measure for $\mu$. We then need to previously fix a zooming set of reference $\Lambda$, which depends both on $\alpha$ and $\delta$.
\end{remark}

\begin{definition}
(Bounded distortion) Given a measure $\mu$ with a Jacobian $J_{\mu}f$, we say that the measure has \textbf{\textit{bounded distortion}} if there exists $\rho > 0$ such that
\[
\bigg{|}\log \frac{J_{\mu}f^{n}(y)}{J_{\mu}f^{n}(z)} \bigg{|} \leq \rho d(f^{n}(y),f^{n}(z)),
\]
for every $y,z \in V_{n}(x)$, $\mu$-almost everywhere $x \in M$, for every zooming time $n$ of $x$.
\end{definition}

The map $f$ with an associated zooming measure with bounded distortion is called a \textbf{\textit{zooming system}}. We then denote a zooming system by $(f,M,\mu,\Lambda)$, where $\Lambda$ is previously fixed, as observed in Remark \ref{fixed}.

The next result for closed maps assures that every ergodic expanding/zooming measure can be lifted to some inducing scheme which is constructed by V. Pinheiro from hyperbolic times.

For the next result we use the notation defined in the section \ref{MarkovMap}, where $F_{i}$ are Markov maps and $\mathcal{P}_{i}$ are partitions. We will state the result in the form which we will use in this work. 

\begin{theorem}[Pinheiro \cite{Pi1},Theorems 1 and D] \label{Pinheiro}
In the setup defined above, that is, a map endowed with a zooming set, there exists a finite Markov structure 
\[
(F_{1},\mathcal{P}_{1}),\dots,(F_{s},\mathcal{P}_{s}),
\] 
such that the probabilities are liftable with uniformly bounded integral of inducing time.
\end{theorem}

Our first main result extends the previous theorem for Markov structures adapted to holes of a special type. See Definition \ref{hole} to clarify what we mean for hole and open system.

\begin{theorema}
\label{A}
Let $(f,M,\mu,\Lambda)$ be a zooming system. Let $r_{0} > 0$ and let $\mathcal{A} = \{B_{r}(p_{i}) \mid i=1,2,\dots,k , r < r_{0}\}$ be a finite open cover of $M$ such that $B_{r}(p_{i}) \cap B_{r/2}(p_{j}) = \emptyset, \text{for all} \, \, i,j \leq k$. We assume that $H \subset M$ satisfies either $H = \emptyset$ or $H$ is a special type of open set chosen such that   
\[
\displaystyle \bigcup_{j=1}^{t} B_{\frac{r_{k_{j}}}{2}}(p_{k_{j}}) \subset H \subset \bigcup_{j=1}^{t} B_{r_{k_{j}}}(p_{k_{j}}),
\]
for some $k_{1},\dots,k_{t} \leq k$. Then if $r_{0}$ is sufficiently small, depending only on $(f,M,\mu,\Lambda)$, there exists a finite Markov structure adapted to the hole $H$.
\end{theorema}

\begin{remark}
The case where $H = \emptyset$ means that the system is actually closed and it reduces to Theorem \ref{Pinheiro}. To extend it, we consider a nested collection instead of a unique nested set. The adaptation here is also an extension to the context of metric spaces of a result in \cite{DT}, where they construct what they call an inducing scheme respecting the hole. In section \ref*{components} we show how $H$ is constructed and we clarify why it is a special type. We still observe that the construction of the Markov structure is made exactly along the same line as in \cite{Pi1} work and the choice of the hole is made by observing what are the open sets for which the Markov structure is adapted.
\end{remark}

\bigskip

\subsection{Pressure, zooming potentials and equilibrium states} For measurable maps we recall the definition of an equilibrium state, given in the Introduction. Given a measurable map $f: M \to M$ on a compact metric space $M$ for which the set $\mathcal{M}_{f}(M)$ of $f$-invariant measures is non-empty (for example, the continuous maps) and a measurable potential $\phi : M \to \mathbb{R}$, we define the \textbf{\textit{pressure}} $P_{f}(\phi)$ as in equation \ref{pressure} and an  \textbf{\textit{equilibrium state}} is an invariant measure that attains the supremum as in equation \ref{equilibrium}.

Denote by $\mathcal{Z}(\Lambda)$ the set of invariant zooming measures supported on $\Lambda$ (in [\cite{Pi1}, Theorem C] it is proved that this set is nonempty under the same hypothesis as ours for the map, that is, a zooming measure with bounded distortion). 

We define a \textbf{\textit{zooming potential}} as a measurable potential $\phi:M \to \mathbb{R}$ such that
\[
\displaystyle  \sup_{\eta \in \mathcal{Z}(\Lambda)^{c}} \bigg{\{} h_{\eta}(f) + \int \phi d\eta \bigg{\}} < \sup_{\eta \in \mathcal{Z}(\Lambda)} \bigg{\{} h_{\eta}(f) + \int \phi d\eta \bigg{\}}.
\]
It is analogous to the definition of expanding potential that can be found in \cite{PV}.

Denote by $h(f)$ the pressure of the potential $\phi \equiv 0$ (or topological entropy), which we call simply \textbf{\textit{entropy}} of $f$, that is,
\[
\displaystyle h(f) := P_{f}(0)  = \sup_{\eta \in \mathcal{M}_{f}(M)} \bigg{\{} h_{\eta}(f) \bigg{\}}.
\]

\begin{example}\label{Base}
	As an example of a zooming potential, if we assume that a potential $\phi_{0} : M \to \mathbb{R}$ is such that
	\[
	\displaystyle  \sup_{\eta \in \mathcal{Z}(\Lambda)^{c}} \bigg{\{} \int \phi_{0} d\eta \bigg{\}} < \sup_{\eta \in \mathcal{Z}(\Lambda)} \bigg{\{} \int \phi_{0} d\eta \bigg{\}},
	\]
	there exists $t_{0} > 0$ such that
	\[
	h(f) < \sup_{\eta \in \mathcal{Z}(\Lambda)} \bigg{\{} \int t_{0}\phi_{0} d\eta \bigg{\}} - \sup_{\eta \in \mathcal{Z}(\Lambda)^{c}} \bigg{\{} \int t_{0}\phi_{0} d\eta \bigg{\}}.
	\]
	By taking $\phi:=t_{0}\phi_{0}$, we obtain
	\[
	\displaystyle  \sup_{\eta \in \mathcal{Z}(\Lambda)^{c}} \bigg{\{}h_{\eta}(f) + \int \phi d\eta \bigg{\}} \leq h(f) + \sup_{\eta \in \mathcal{Z}(\Lambda)^{c}} \bigg{\{} \int \phi d\eta \bigg{\}}< \sup_{\eta \in \mathcal{Z}(\Lambda)} \bigg{\{} \int \phi d\eta \bigg{\}} \leq \sup_{\eta \in \mathcal{Z}(\Lambda)} \bigg{\{}h_{\eta}(f) + \int \phi d\eta \bigg{\}}.
	\]
\end{example}

\begin{example} \label{Dirac}
	Let us suppose that there exists a fixed point $x_{0} \in \Lambda$ and a potential $\phi_{0} : M \to \mathbb{R}$ such that $\phi_{0}(x) < \phi_{0}(x_{0})$ for all $x \in M \backslash \{x_{0}\}$. It is clear that the Dirac probability $\delta_{x_{0}}$ supported on $x_{0}$ is a zooming measure. Moreover, if we assume that
	\[
	\displaystyle  \sup_{\eta \in \mathcal{Z}(\Lambda)^{c}} \bigg{\{} \int \phi_{0} d\eta \bigg{\}} < \int \phi_{0} d \delta_{x_{0}} = \sup_{\eta \in \mathcal{Z}(\Lambda)} \bigg{\{} \int \phi_{0} d\eta \bigg{\}},
	\]
	then, by the previous example, we can find $t_{0} > 0$ such that $\phi:= t_{0}\phi_{0}$ is a zooming potential.
\end{example}

\begin{example}\label{Periodic}
	To generalize the previous example, let us suppose that there exists a periodic point $x_{0} \in \Lambda$ with period $k \geq 1$ and a potential $\phi_{0} : M \to \mathbb{R}$ such that $\phi_{0}(x) < \phi_{0}(f^{i}(x_{0})), i=0,1,\dots,k-1$ for all $x \in M \backslash \{f^{i}(x_{0}) \mid i=0,1,\dots,k-1\}$. It is clear that the average of Dirac probabilities $\mu_{0}:=(1/k)\sum_{i=0}^{k-1}\delta_{f^{i}(x_{0})}$ supported on the orbit of $x_{0}$ is a zooming measure. Moreover, if we assume that
	\[
	\displaystyle  \sup_{\eta \in \mathcal{Z}(\Lambda)^{c}} \bigg{\{} \int \phi_{0} d\eta \bigg{\}} < \int \phi_{0} d \mu_{0} = \sup_{\eta \in \mathcal{Z}(\Lambda)} \bigg{\{} \int \phi_{0} d\eta \bigg{\}},
	\]
	then, we can find $t_{0} > 0$ such that $\phi:= t_{0}\phi_{0}$ is a zooming potential.
\end{example}

\begin{remark}
	We observe that in the previous examples we have for every invariant probability $\mu$
	\[
	\int \phi_{0} d\mu \leq \int \phi_{0} d\delta_{x_{0}} \,\, \text{and} \,\, \int \phi_{0} d\mu \leq \int \phi_{0} d\mu_{0}, \,\, \text{respectively}.
	\]
	What we assume is that the supremum of the integrals over all non zooming measures is strictly less than the integral with respect to the probabilities $\delta_{x_{0}}$ and $\mu_{0}$, respectively.
\end{remark}

For open systems, we consider another type of pressure and equilibrium states as follows.

\begin{definition}(Open pressure and open equilibrium states)
   Given an open system $(f,M,H)$ with a hole $H$ (see Definition \ref{hole}), we define the \textbf{\textit{open pressure}} as
   \begin{equation}\label{open pressure}
   	\displaystyle P_{f,H}(\phi) : = \sup_{\eta \in \mathcal{M}_{f}(M,H)} \bigg{\{} h_{\eta}(f) + \int \phi d\eta \bigg{\}},
   \end{equation}
    where $\mathcal{M}_{f}(M,H)$ is the set of $f$-invariant measures $\eta$ such that $\eta(H) = 0$. Moreover, we define an \textbf{\textit{open equilibrium state}} as an invariant measure $\mu \in \mathcal{M}_{f}(M,H)$ such that
    \begin{equation}\label{open equilibrium}
    \displaystyle h_{\mu}(f) + \int \phi d\mu = \sup_{\eta \in \mathcal{M}_{f}(M,H)} \bigg{\{} h_{\eta}(f) + \int \phi d\eta \bigg{\}},
    \end{equation} 	
\end{definition}

\begin{remark}
	We readily have that $P_{f,H}(\phi) \leq P_{f}(\phi)$. Also, in the case where the system is closed, that is, where $H = \emptyset$, the open pressure reduces to the pressure and the open equilibrium state reduces to the equilibrium state, both previously defined.
\end{remark}

\begin{remark}
	The definitions \ref{open pressure} and \ref{open equilibrium} are more appropriate for the open systems because they focus on measures which give full mass to the survivor set $M^{\infty}$ (see Definition \ref{survivor}). In fact, if a measure gives full mass to the survivor set $M^{\infty}$, then it gives null mass to the hole $H$. Also, since the survivor set $M^{\infty}$ is invariant, if an $f$-invariant measure gives null mass to the hole $H$, then it gives full mass to the survivor set $M^{\infty}$. The study of equilibrium states is now concentrated at the survivor set $M^{\infty}$.
\end{remark}

\begin{definition}
(Backward separated map) We say that a map $f:M \to M$ is \textbf{\textit{backward separated}} if for every finite set $F \subset M$ we have
\[
\displaystyle d\Bigg{(}F, \bigcup_{j=1}^{n} f^{-j}(F) \backslash F \Bigg{)} > 0, \text{for all} \, \,\, n \geq 1.
\]
\end{definition}
Observe that if $f$ is such that $\sup\{\# f^{-1} (x) \mid x \in M \}< \infty$, then $f$ is backward separated.

Denote by $\omega_{f,+}(z)$ the set of accumulation points of the orbit $\{f^{n}(z)\}_{n}$ at zooming times.

In order to make a zooming system open, we consider the hole obtained by Theorem \ref{A}. Now, we state our main result on equilibrium states. With a quite general setup, we obtain existence and uniqueness with no requirement of transitivity!

\begin{theorema}
\label{B}
Given a measurable open zooming system $f:M \to M$ which is backward separated, if the contraction $(\alpha_{n})_{n}$ satisfies $\alpha_{n}(r) \leq ar$ for some $a \in (0,1)$, every $n \in \mathbb{N}$ and every $r \in [0,+\infty)$ (Lipschitz, for example), with zooming set  $\Lambda$ (previously fixed) and hole $H$ given by Theorem \ref{A}. 

\begin{itemize}

\item If $\phi:M \to \mathbb{R}$ is zooming potential with  finite pressure $P_{f}(\phi)$ and locally Hölder induced potential ($\phi$ Hölder, for example), then there are \textbf{finitely many}  ergodic equilibrium states and they are zooming measures.  

\item If the zooming set $\Lambda$ is not dense in $M$, we can choose the hole such that $\Lambda \cap H = \emptyset$ and obtain the equilibrium states giving full mass to the survivor set $M^{\infty}$. Then, in this case, the equilibrium states are open.

\item Afterwards, we establish \textbf{uniqueness} of equilibrium state if there exist infinitely many ergodic zooming measures whose supports are pairwise disjoints, with no requirement of transitivity.
\end{itemize}
\end{theorema}

\begin{remark}\label{remark} We observe that
\begin{itemize}	
\item Theorem \ref{B} includes a similiar result that can be seen in \cite{AOS} and the proof is along the same lines. Later on, we will see that hyperbolic potentials are equivalent to continuous zooming potentials and we can use Theorem \ref{A} to obtain the same result for open non-uniformly expanding maps. Theorem \ref{B} reduces to closed system if $H = \emptyset$. Theorem \ref{B} is also similar to some of the main results of \cite{PV}. The main difference is the technique and we also include zooming systems with nonexponential Lipschitz contraction.

\item The coding of the system by using the Markov structure and considering the system as closed gives finitely many ergodic equilibrium states which are zooming measures. Since the Markov structure is adapted to the hole $H$, when the zooming set $\Lambda$ is disjoint from $H$, we also obtain the property that the equilibrium states are supported on the survivor set $M^{\infty}$, since the measures are zooming (giving full mass to $\Lambda$ and, so, to $M^{\infty}$). As a consequence, they give null mass to the hole $H$ and we have $P_{f,H}(\phi) = P_{f}(\phi)$, which implies that the equilibrium states are open equilibrium states when the system is open and $\Lambda \cap H = \emptyset$. 

\item For the Lipschitz contraction $(\alpha_{n})_{n}(r) = a_{n}r$ we have $\sum_{n=1}^{\infty} a_{n} < \infty$ and so $a_{n} \to 0$. Then, we can take $a = \max\{a_{n}\}$.

\item If $\Lambda$ is dense in $M$ and there exists a not dense forward invariant subset $\Lambda_{0} \subset \Lambda$ such that $\mu(\Lambda \backslash \Lambda_{0}) = 0$, we can consider the open system with respect to the zooming set $\Lambda_{0}$.

\item To obtain the uniqueness, we assume that there exist infinitely many ergodic zooming measures to take ergodic zooming measures which are not equilibrium states. It is possible because we have finiteness of equilibrium states. 
\end{itemize}
\end{remark}

\subsection{Pseudo-conformal measures and pseudo-geometric potentials}
Given a measure $\mu$ on $M$, its \textbf{\textit{Jacobian}} is a function $J_{\mu}f : M \to [0,+ \infty)$ such that 
\[
\mu(f(A)) = \int_{A} J_{\mu}f d\mu
\] 
for every $A$ \textbf{\textit{domain of injectivity}}, that is, a measurable set such that $f(A)$ is measurable and $f_{A}: A \to f(A)$ is a bijection.

The class of conformal measures is among the main measures we can choose in order to study the thermodynamic formalism of a given dynamical system $f:M \to M$. They are measures $\eta$ with a Jacobian $J_{\eta}f$ of the type $J_{\eta}f = e^{-\phi}$, where $\phi: M \to \mathbb{R}$ is a potential. It means that 
\[
\eta(f(A)) = \int_{A} e^{-\phi} d\eta,
\]
for every measurable domain of injectivity $A \subset M$. The potentials we will consider are the so-called \textbf{\textit{pseudo-geometric potentials}}, defined as follows
\[
\phi_{t}(x) =
\begin{cases}
 -t \log J_{\mu}f(x), & \text{if }   J_{\mu}f(x) \neq 0;\\
 0 , & \text{if }  J_{\mu}f(x) = 0.\\
\end{cases}
\]

We require that the Jacobian $J_{\mu}f$ is bounded above and the set of the points $x \in M$ where $J_{\mu}f(x) = 0$ has zero measure $\mu$. It means that
\[
\int e^{-\phi_{t}} d\mu = \int_{J_{\mu}f(x) \neq 0} e^{-\phi_{t}} d\mu =  \int_{J_{\mu}f(x) \neq 0} e^{t \log J_{\mu}f(x)} d\mu = \int (J_{\mu}f)^{t} d\mu
\]
and we call the measure $\mu$ \textbf{\textit{pseudo-conformal}}.

For the zooming reference measure $\mu$, Pinheiro showed in [\cite{Pi1}, Theorem C] that there are finitely many ergodic absolutely continuous measure with respect to $\mu$. We denote by $\mathbb{A}$ the set of such measures. Fix $\mu_{0} \in \mathbb{A}$ we know that $J_{\mu}f$ is also a Jacobian for $\mu_{0}$. From the definition of Jacobian, since $\mu_{0}$ is an invariant measure, it follows that $J_{\mu}f(x) \geq 1$, $\mu_{0} \,\, \text{a.e.} \,\, x \in M$. If fact, we have for domain of injectivity $A$
\[
\mu_{0}(A) \leq  \mu_{0}(f(A))=  \int_{A} J_{\mu}f d\mu_{0} \implies J_{\mu}f(x) \geq 1, \mu_{0} \,\, \text{a.e.} \,\, x \in A. 
\]
Since the zooming set $\Lambda$ has full measure $\mu_{0}$ and every zooming pre-ball is a domain of injectivity, we can cover a set of full measure with domains of injectivity. We conclude that $J_{\mu} f(x) \geq 1$ a.e. $x \in M$. We cannot have $J_{\mu} f(x) \equiv 1$ a.e. $x \in M$. So,
\[
\displaystyle \int \log J_{\mu} f d\mu_{0} > 0.
\]
We assume that $h(f) < \infty$ and let
\[
\displaystyle t_{0} : = \max_{\mu_{0} \in \mathbb{A}} \bigg{\{} \frac{h(f)}{-\int \log J_{\mu} f d\mu_{0}} \bigg{\}} \leq 0.
\]
The pressure $P_{f}(\phi_{t})$ is finite in our setting, since the potential $\phi_{t}$ is bounded above.

\begin{definition}
	We say that a map $f:M \to M$ is \textbf{\textit{topologically exact}} if for every open set $V \subset M$ there exists $k \in \mathbb{N}$ such that $f^{k}(V) = M$. 	
\end{definition}

\begin{theorema}
\label{C}
Given a measurable open zooming system $f:M \to M$ which is backward separated, if the contraction $(\alpha_{n})_{n}$ satisfies $\alpha_{n}(r) \leq ar$ for some $a \in (0,1)$, every $n \in \mathbb{N}$ and every $r \in [0,+\infty)$ (Lipschitz, for example), with zooming set  $\Lambda$ (previously fixed) and hole $H$ given by Theorem \ref{A}. 

\begin{itemize}
	
	\item for $t < t_{0}$ the potential $\phi_{t} = -t \log J_{\mu}f$ is zooming and the induced potential is locally Hölder. There are \textbf{finitely many}  ergodic equilibrium states and they are zooming measures. 
	
	\item If the zooming set $\Lambda$ is not dense in $M$, we can choose the hole such that $\Lambda \cap H = \emptyset$ and obtain the equilibrium states giving full mass to the survivor set $M^{\infty}$. Then, in this case, the equilibrium states are open.
	
	\item Afterwards, we establish \textbf{uniqueness} of equilibrium state if there exist infinitely many ergodic zooming measures whose supports are pairwise disjoints, with no requirement of transitivity.
		
	\item if the map is topologically exact, there exists a pseudo-conformal measure for the potential $\phi_{t} - P_{f}(\phi_{t})$;
\end{itemize} 
\end{theorema}

\begin{remark}
	Here it holds the same explanation as Remark \ref{remark} for Theorem \ref{B}. Moreover, we observe that a similar result can be found in \cite{IT1} for conformal measures and geometric potentials in the context of one-dimensional dynamics. We extend it to the context of metric spaces and open systems, with the proof along the same lines. The main ingredient here is the Markov structure obtained in our Theorem \ref{A} and the theory of O. Sarig for symbolic dynamics.
\end{remark}

Now, we define Markov Structure and what we mean for it to be adapted to a hole.

\subsection{Markov Maps and Markov Structures} \label{MarkovMap}

We recall the definitions of Markov partition, Markov map and induced Markov map (See \cite{Pi1} for details). 

Consider a measurable map $f:M \to M$ defined on the metric space $M$ endowed with the Borel $\sigma$-algebra.

\begin{definition}
(Markov partition). Let $f:U \to U$ be a measurable map defined on a Borel set $U$ of a compact, separable metric space $M$. 
A countable collection $\mathcal{P} = \{P_{1}, P_{2}, P_{3}, \dots\}$ of Borel subsets of $U$ is called a \textbf{\textit{Markov partition}} if
\begin{enumerate}
  \item[(1)]$int(P_{i}) \cap int(P_{j}) = \emptyset$, if $i \neq j$;
  \item[(2)] $f_{\mid_{P_{i}}}$ is a homeomorphism and it can be extended to a homeomorphism sending $\overline{P_{i}}$ onto $\overline{f(P_{i})}$;
  \item[(3)] if $f(P_{i}) \cap int(P_{j}) \neq \emptyset$ then $f(P_{i}) \supset int(P_{j})$;
  \item[(4)] $\# \{f(P_{i}) \mid i \in \mathbb{N}\} < \infty$;
  \item[(5)] $\displaystyle \lim_{n \to \infty} diam(\mathcal{P}_{n}(x)) = 0, \,\, \text{for all} \, \,\, x \in \cap_{n \geq 0} f^{-n}(\cup_{i}P_{i})$,
\end{enumerate}
where $\mathcal{P}_{n}(x) = \{y \mid \,\, \mathcal{P}(f^{j}(y)) = \mathcal{P}(f^{j}(x)), \,\, \text{for all} \, \,\, 0 \leq j \leq n\}$ and 
$\mathcal{P}(x)$ denotes the element of $\mathcal{P}$ that contains $x$.
\end{definition}

\begin{definition}
(Markov map). The pair $(F,\mathcal{P})$, where $\mathcal{P}$ is a Markov partition of $F:U \to U$, is called a \textbf{\textit{Markov map}} defined on $U$. 
If $F(P) = U, \,\, \text{for all} \, \,\, P \in \mathcal{P}, (F, \mathcal{P})$ is called a \textbf{\textit{full Markov map}}.
\end{definition}

Note that if $(F,\mathcal{P})$ is a full Markov map defined on an open set $U$ then the elements of $\mathcal{P}$ are open sets 
(because $F(P)=U$ and $F\mid_{P}$ is a homeomorphism, $\text{for all} \, \,\, P \in \mathcal{P}$).

\begin{definition}
(Induced Markov map). A Markov map $(F,\mathcal{P})$ defined on $U$ is called an \textbf{\textit{induced Markov map}} for $f$ on $U$ if there exists a function 
$R:U \to \mathbb{N}=\{0,1,2,3,\dots\}$(called \textbf{\textit{inducing time}}) such that $\{R \geq 1\} = \cup_{P \in \mathcal{P}} P$, $R\mid_{P}$ is constant
$\text{for all} \, \, P \in \mathcal{P}$ and $F(x)=f^{R(x)}(x), \text{for all} \, \,\, x \in U$. We also denote  $R_{k}:=R(P_{k})$.
\end{definition}

If an induced Markov map $(F,\mathcal{P})$ is a full Markov map, we call $(F,\mathcal{P})$ an \textbf{\textit{induced full Markov map}}. 
We will also call the triple $(F,\mathcal{P})$ an \textbf{\textit{inducing scheme}}.

\subsection{Open Dynamics}

For the classical dynamical systems the phase spaces are invariant and called \textbf{\textit{closed}}. When we consider systems where the phase space is not invariant, they are called \textbf{\textit{open}}. It is done by considering holes in a closed dynamical system.

\begin{definition}\label{hole}
(Hole and open system). Given a dynamical system $f:M \to M$, we say that an open subset $\displaystyle H \neq M$ with finitely many connected components is a \textbf{\textit{hole}} for the system and the system $(f,M,H)$ is  \textbf{\textit{open}}.
\end{definition}

\begin{definition}\label{survivor}
(Survivor set). Given an open dynamical system $f:M \to M\backslash H$ with hole $H$, we say that the set
\[
\displaystyle M^{\infty} : = \bigcap_{i=0}^{\infty} f^{-i}(M \backslash H)
\]
is the \textbf{\textit{survivor set}} of the system. It means the set of point that never pass through the hole.
\end{definition}

\begin{remark}
	We observe that for that case where $\Lambda \cap H = \emptyset$, we have that $\Lambda \subset M^{\infty}$ because the zooming set is invariant. As a consequence, the survivor set $M^{\infty}$ has full measure with respect to any zooming measure.  
\end{remark}

\begin{definition}
(Inducing scheme adapted to the hole). Given a hole $H \subset M$, we say that an inducing scheme is  $(F,\mathcal{P})$ \textbf{\textit{adapted to the hole}} $H$ 
if we have $f^{k}(P) \cap H \neq \emptyset \Rightarrow f^{k}(P) \subset H,\, \text{for all} \, \,\, P \in \mathcal{P}, \text{for all} \, \,\, 0 \leq k < R(P)$. 
\end{definition}

\begin{definition}
(Measure liftable to a inducing scheme). Given an inducing scheme $(F,\mathcal{P})$ and an invariant probability $\mu$, we say that $\mu$ is \emph{liftable} to $(F, \mathcal{P})$ if there exists a measure 
$\overline{\mu}$ on $U$ such that for every measurable set $A \subset M$, 
\[
\mu(A) = \sum_{k=1}^{\infty}
\sum_{j=0}^{R_{k} - 1} \overline{\mu}(f^{-j}(A) \cap P_{k}).
\]
\end{definition}

\begin{definition}
(Markov structure adapted to the hole). A \textbf{\textit{Markov structure}} for a set $\Lambda$ is defined as a collection at most countable of inducing schemes $(F_{i},\mathcal{P}_{i})$ such that every ergodic probability $\mu$ with $\mu(\Lambda)=1$ is liftable to some of them.

Given a hole $H \subset M$, we say that a Markov structure is  $(F,\mathcal{P})$ \textbf{\textit{adapted to the hole}} $H$ 
if inducing scheme is adapted to it. 
\end{definition}

\section{Preliminary Results}\label{Preliminary}

In order to construct an inducing scheme that is adapted to a hole of a special type, we need the notions of \textit{linked sets}, \textit{nested collections} and 
\textit{zooming times}, introduced by Pinheiro in \cite{Pi1}. Zooming times generalize the notion of \textit{hyperbolic times}, which are fundamental for our construction.
The elements of the partition will be regular pre-images of a certain open set and the hole will be obtained from small balls by deleting linked regular pre-images 
of the considered balls. 

\subsection{Nested Collections}

We recall some definitions and results, that will help us to show that the induced map we will build it adapted to a special type of hole. All the statements and proofs can be found in \cite{Pi1}.

\begin{definition}(Linked sets).
We say that two open sets $A$ and $B$ are \textbf{\textit{linked}} if both $A-B$ and $B-A$ are not empty. 
\end{definition}

We introduce the following notation.

\textbf{Notation}: We write $A \leftrightarrow B$ to mean that $A$ and $B$ are linked and $A \not \leftrightarrow B$ to mean that $A$ and $B$ are not linked.

\begin{definition}
(Regular pre-images). Given $V$ an open set, we say that $P$ is a \textbf{\textit{regular pre-image}} of order $n$ of $V$ if $f^{n}$ sends $P$ homeomorphically onto $V$. 
Denote by $ord(P)$ the order of $P$ (with respect to $V$).
\end{definition}

Let us fix a collection $\mathcal{E}_{0}$ of open sets. For each $n$ we consider $\mathcal{E}_{n}(V)$ as the collection of regular pre-images of order $n$ of $V$.
Set $\mathcal{E}_{n} = (\mathcal{E}_{n}(V))_{V \in \mathcal{E}_{0}}$. We call the sequence $\mathcal{E} = (\mathcal{E}_{n})_{n}$ 
a \textbf{\textit{dynamically closed family}} of regular pre-images. We observe that $f^{k}(E) \in \mathcal{E}_{n-k}, \text{for all} \, \,\, E \in \mathcal{E}_{n}, \text{for all} \, \,\, 0 \leq k \leq n$. 

Let $\mathcal{E} = (\mathcal{E}_{n})_{n}$ be a dynamically closed family of regular pre-images. A set $P$ is called an $\mathcal{E}$-pre-image of a set $W$ 
if either there is $n \in \mathbb{N}$ and $Q \in \mathcal{E}_{n}$ such that $\overline{W} \subset f^{n}(Q)$ and $P = f_{|Q}^{-n}(W)$ or $W=V \in \mathcal{E}_{n}$.

\begin{definition}
(Nested sets). An open set $V$ is called $\mathcal{E}$\textbf{\textit{-nested}} if it is not linked with any $\mathcal{E}$-pre-image of it.
\end{definition}

\begin{definition}
(Nested collections). A collection $\mathcal{A}$ of open sets is called an $\mathcal{E}$\textbf{\textit{-nested collection}} of sets if every $A \in \mathcal{A}$ 
is not linked with any $\mathcal{E}$-pre-image of an element of $\mathcal{A}$ with order bigger than zero. Precisely, if $A_{1} \in \mathcal{A}$ 
and $P$ is an $\mathcal{E}$-pre-image of some $A_{2} \in \mathcal{A}$, then either $A_{1} \not \leftrightarrow P$ or $P = A_ {2}$.
\end{definition}

It follows from the definition of an $\mathcal{E}$-nested collection of sets that every sub-collection of an $\mathcal{E}$-nested collection
is also an $\mathcal{E}$-nested collection. In particular, each element of an $\mathcal{E}$-nested collection is an $\mathcal{E}$-nested set.

\begin{lemma}
\label{Main}
(Main property of a nested collection). If $\mathcal{A}$ is an $\mathcal{E}$-nested collection of open sets and $P_{1}$ and $P_{2}$ are $\mathcal{E}$-pre-images
of two elements of $\mathcal{A}$ with $ord(P_{1}) \neq ord(P_{2})$ then $P_{1} \not \leftrightarrow P_{2}$.
\end{lemma}

\begin{corollary}
(Main property of a nested set). If $V$ is an $\mathcal{E}$-nested set and $P_{1}$ and $P_{2}$ are $\mathcal{E}$-pre-images of $V$ then 
$P_{1} \not \leftrightarrow  P_{2}$. Furthermore,

\begin{enumerate}
  \item[(1)] if $P_{1} \cap P_{2} \neq \emptyset$ then $ord(P_{1}) \neq ord(P_{2})$;
  
  \item[(2)] if $\displaystyle P_{1} \subset_{\neq} P_{2}$ with $ord(P_{1}) < ord(P_{2})$ then $V$ is contained in an $\mathcal{E}$-pre-image of itself 
  with order bigger than zero, $f^{ord(P_{2}) - ord(P_{1})}(V) \subset V$.
\end{enumerate}
\end{corollary}

\subsection{Constructing Nested Sets and Collections}

Let $\mathcal{A}$ be a collection of connected open subsets such that the elements of $\mathcal{A}$ are not contained in any $\mathcal{E}$-pre-image 
of order bigger than zero of an element of $\mathcal{A}$.

\begin{definition}
(Chains of pre-images). A finite sequence $\mathcal{K} = \{P_{0},P_{1}, \dots, P_{n}\}$ of $\mathcal{E}$-pre-images of $\mathcal{A}$ 
is called a \textbf{\textit{chain}} of $\mathcal{E}$-pre-images beginning in $A \in \mathcal{A}$ if

\begin{itemize}

  \item $0 < ord(P_{0}) \leq ord(P_{1}) \leq \dots \leq ord(P_{n})$;
  
  \item $A \leftrightarrow  P_{0}$;
  
  \item $P_{j-1} \leftrightarrow P_{j}, \, 1 \leq j \leq n$; 
  
  \item $P_{i} \neq P_{j}$ if $i \neq j$.

\end{itemize}

\end{definition}

For each $A \in \mathcal{A}$ define the open set
\begin{equation}\label{star}
\displaystyle A^{\star} = A \backslash \overline{\bigcup_{(P_{j})_{j} \in ch_{\mathcal{E}}(A)} \bigcup_{j} P_{j}}
\end{equation}
where $ch_{\mathcal{E}}(A)$ is the set of chains.

\begin{proposition}
\label{nested}
(An abstract construction of a nested collection). For each $A \in \mathcal{A}$ such that $A^{\star} \neq \emptyset$ choose a connected component 
$A'$ of $A^{\star}$. If $\mathcal{A'} = \{A'\mid A \in \mathcal{A} \,\, and \,\, A^{\star} \neq \emptyset\}$ is not an empty collection then $\mathcal{A'}$ 
is an $\mathcal{E}$-nested collection of sets.
\end{proposition}

Let us introduce the following notation.

\textbf{Notation}: We denote by $\mathcal{E_{Z}}=
(\mathcal{E_{Z}}_{,n})_{n}$ the collection of all $ (\alpha,\delta)$-zooming pre-balls, where $\mathcal{E_{Z}}_{,n}=\{
V_{n}(x) \mid x \in Z_{n}(\alpha,\delta,f)\}$. Observe that this collection is a dynamically closed family of pre-images.

With the notation of Corollary $\ref{nestcoll}$, let $A_{i}$ be contained in a zooming ball for all $i=0,1,2,\dots,k$. This corollary implies that we have
$\mathcal{A}'$ an $\mathcal{E_{Z}}$-nested collection, if the chains are small enough.

\begin{definition}
(Zooming nested collection and ball). We call $\mathcal{A}'$ an $(\alpha,\delta)$\textbf{\textit{-zooming nested collection}} and each element of an $(\alpha,\delta)$-zooming nested collection is a \textbf{\textit{zooming nested ball}}.  
\end{definition}

\section{Markov Structures Adapted to Special Holes}\label{Markov}

In this section, we proceed with the proof of Theorem $\ref{A}$.

We will use the notation and results of the previous sections to prove some preliminary results. We will extend the technique of Pinheiro in \cite{Pi1}.

\subsection{Existence of Nested Collections}

The following Lemma is proved in \cite{Pi1} for the case of a collection with one ball as corollary of Proposition \ref{nested}. 
We prove it here for a collection with finitely many open sets, pairwise disjoint.

\begin{lemma}
\label{nestcoll}
(Existence of nested collections). Let $\epsilon \in (0,1)$ and $\mathcal{A} = \{A_{0},A_{1},A_{2},\dots,A_{k}\}$ be a finite collection of pairwise disjoint open sets.
Given $p_{i} \in A_{i}, i=0,1,2,\dots,k$, set $r_{i} = d(p_{i}, \partial A_{i}), i=0,1,2,\dots,k$. If we have

\begin{itemize}
\item $f^{n}(A_{i}) \not\subset A_{j},  \text{for all} \, \,\, n \geq 1, i,j=0,1,2,\dots,k$;

\item Every chain of $\mathcal{E}$-pre-images of $\mathcal{A}$ has diameter less than $m_{0} = \min\{\epsilon r_{i} \mid i=0,1,\dots,k\}$;
\end{itemize}

Then the set $A_{i}^{\star}$ contains the ball $B_{r_{i}(1-\epsilon)}(p_{i}), i=0,1,2,\dots,k$. Moreover, setting $A_{i}'$ as the connected component of $A_{i}^{\star}$ 
that contains $p_{i}$, we have that $\mathcal{A}'=\{A_{i}' \mid i=0,1,2,\dots,k\}$ is an $\mathcal{E}$-nested collection.

\end{lemma}

\begin{proof}
Since $f^{n}(A_{i}) \not\subset A_{j},  \text{for all} \, \,\, n \geq 1, i,j=0,1,2,\dots,k$, we have that $A_{i}$ is not contained in any $\mathcal{E}$-pre-image of $\mathcal{A}$
(with order bigger than zero) $i=0,1,2,\dots,k$. Let $\Gamma_{A_{i}}$ be the collection of all chains intersecting $\partial A_{i}$. If $(P_{j})_{j} \in \Gamma_{A_{i}}$, 
then $\displaystyle \bigcup_{j} P_{j}$ is a connected open set intersecting $\partial A_{i}$ with diameter less than $m \leq \epsilon r_{i}$. 
So, $\displaystyle \bigcup_{j} P_{j} \subset B_{\epsilon r_{i}}(\partial A_{i}), \text{for all} \, (P_{j})_{j} \in \Gamma_{A_{i}}$. 
As a consequence, we obtain $\displaystyle A_{i}^{\star} = A_{i} \backslash \overline{\bigcup_{(P_{j})_{j} \in ch_{\mathcal{E}}(A_{i})} \bigcup_{j} P_{j}} \supset A_{i}
\backslash \overline{B_{\epsilon r_{i}}(\partial A_{i})} \supset B_{r_{i}(1 - \epsilon)}(p_{i})$. Then, $A_{i}^{\star} \supset B_{r_{i}(1-\epsilon)}(p_{i})$. 
Setting $A_{i}'$ the connected component of $A_{i}^{\star}$ that contain $p_{i}$, by the previous proposition we have that $\mathcal{A}'$ 
is an $\mathcal{E}$-nested collection.   
\end{proof}

\subsection{Existence of Zooming Nested Collections}

The following lemma gives a sufficient condition to guarantee that the chains are small enough, in order to show that $\mathcal{A}'$ in the previous section is an 
$(\alpha,\delta)$-zooming nested collection. It is proved in \cite{Pi1} for the case of a nested set.

\begin{lemma}
\label{zoomnestcoll}
(Existence of zooming nested collection). Let $\epsilon \in (0,1)$ and  $M_{0} = \max\{diam(A_{i}) \mid i=0,1,2,\dots,k\}$ and $m_{0}=\min\{\epsilon r_{i} \mid i=0,1,2,\dots,k \}$.
We also suppose that $diam(A_{i}) > \alpha_{n} (diam(A_{j})), i,j=0,1,2,\dots,k, \text{for all} \,\, n \in \mathbb{N}$.
\begin{itemize}
\item If we have $\sum_{n \geq 1} \alpha_{n}(M_{0}) \leq m_{0}$, then $A_{i}^{\star}$ is well defined and  $A_{i}^{\star} \supset B_{(1-\epsilon)r_{i}}(p_{i}), i=0,1,2,\dots,k$. 

\item If $f$ is backward separated, $\sup_{r>0} \sum_{n \geq 1} \alpha_{n}(r)/r < \infty$ and $M_{0}$ is such that $M_{0}\epsilon/2 \leq m_{0}$, then $A_{i}^{\star}$ is well defined
and $A_{i}^{\star} \supset B_{(1- \epsilon)r_{i}}(p_{i}), i=0,1,2,\dots,k$.
\end{itemize}
\end{lemma}

\begin{proof}
Firstly, observe that $f^{n}(A_{i}) \not\subset A_{j},  \text{for all} \, \,\, n \geq 1, i,j=0,1,2,\dots,k$, since $diam(A_{i}) > \alpha_{n} (diam(A_{j})), i,j=0,1,2,\dots,k, 
\text{for all} \, n \in \mathbb{N}$. Moreover, every chain  has diameter less than $\sum_{n \geq 1} \alpha_{n}(a_{n})$ for some $(a_{m})_{m} \in \{diam(A_{0}),diam(A_{1}),
diam(A_{2}),\dots,diam(A_{k})\}^{\mathbb{N}}$ and also, $\sum_{n \geq 1} \alpha_{n}(a_{n}) < \sum_{n \geq 1} \alpha_{n}(M_{0}) \leq m_{0}$. 
By Lemma \ref{nestcoll}, the first part is done.

For the second part, we are assuming that $\sup_{r>0} \sum_{n \geq 1} \alpha_{n}(r)/r < \infty$, then there exists $n_{0} \in \mathbb{N}$ such that $\sum_{n \geq n_{0}} \alpha_{n}(M_{0})/M_{0} 
< \epsilon/2$. If $f$ is backward separated, let $\gamma > 0$ such that 
\[
d(F_{0}, \cup_{j=1}^{n_{0}} f^{-j}(F_{0}) \backslash F_{0}) > \gamma,
\]
where $F_{0}=\{p_{0},p_{1},\dots,p_{k}\}$. Set $r_{\gamma} = \frac{1}{6}\min\{\epsilon,\gamma\}$. Suppose that $M_{0} < 2r_{\gamma}$. 
Note that if $j < n_{0}$ then $A_{i} \cap P = \emptyset, \,\, \text{for all} \, P \in \mathcal{E_Z}_{,j}$ 
(because $P \cap \big{(} \cup_{j=1}^{n_{0}} f^{-j}(F_{0}) \backslash F_{0} \big{)} \neq \emptyset$ and $diam(P) < diam(A_{i}) < 2r_{\gamma} 
< \gamma, \epsilon, \,\, \text{for all} \, \,\, i)$. Thus, every chain of $\mathcal{E_{Z}}$-pre-images of $\mathcal{A}$ begins with a pre-image of order bigger than $n_{0}$. 
Observe that the diameter of any chain is smaller than $\sum_{n \geq n_{0}} \alpha_{n}(M_{0}) < M_{0} \epsilon/2 \leq m_{0} \leq \epsilon r_{i}, i=0,1,2,\dots,k$ and, 
as a chain intersects the boundary $A_{i}$, we can conclude that the chain cannot intersect $B_{(1-\epsilon)r_{i}}(p_{i})$. 
So, we obtain $A_{i}^{\star} \supset B_{(1- \epsilon)r_{i}}(p_{i})$.
\end{proof}

\subsection{Constructing a Markov Structure} \label{components}

Let $f:M \to M$ be a zooming system. If a zooming time is also a return time we call it a \textbf{\textit{zooming return}}.

Let $r < \frac{\delta}{4}$ and  $\{A_{i} = B_{r}(p_{i}) \mid i=1,2,\dots,k\}$ be a finite covering of $M$ such that $B_{r}(p_{i}) \cap B_{r/2}(p_{j}) = \emptyset, \text{for all} \, \, i,j \leq k$. By considering $A_{1},A_{2},\dots,A_{k}$ with diameters $diam(A_{i})=2r, i=0,1,2,\dots,k$,
we can take $M_{0} = 2r$, in the Lemma \ref{zoomnestcoll} to obtain an $(\alpha,\delta)$ -zooming nested collection $\mathcal{A}' = \{A_{1}',A_{2}',\dots,A_{k}'\}$ 
such that  $B_{\frac{r}{2}}(p_{i}) \subset A_{i}' \subset B_{r}(p_{i}), \text{for all} \, i \leq k$. We choose some sets $A_{k_{1}}',A_{k_{2}}',\dots,A_{k_{t}}' \in \mathcal{A}'$ and construct the induced full Markov map in $\cup_{i=1}^{k} A_{i}'$, that is
adapted to the hole $H=\displaystyle \cup_{j=1}^{t} A_{k_{j}}'$ or $H = \emptyset$. Note that $diam(A_{i}') \leq \frac{\delta}{2}$ and $\displaystyle \cup_{j=1}^{t} B_{\frac{r}{2}}(p_{k_{j}}) \subset H \subset \cup_{j=1}^{t} B_{r}(p_{k_{j}})$.

\begin{remark}
	We observe that we need the hole $H$ as a union of $\mathcal{E}$-nested sets to guarantee the adaptation of the Markov structure. If we had a general union of open sets, for an element $P$ of the partition, it might happen that $f^{n}(P) \cap H \neq \emptyset$ without having $f^{n}(P) \subset H$. Moreover, in the equation \ref{star} we need the contraction of the zooming balls to guarantee the shrinking of the pre-balls in order to have $P_{j} \subset_{\neq} A$ and $A^{\star} \neq \emptyset$. Because of it we need to consider a union of $\mathcal{E}$-nested sets instead considering $H^{\star}$ for a general open set $H$.
\end{remark}
\begin{claim}
We claim that any inducing scheme constructed in some $A'_{i}, i \leq k$, is adapted to the hole $H$.
\end{claim}

Let $A_{0}' \in \mathcal{A}'$ and $\mathcal{E_{Z}}$ be the collection of zooming pre-balls. Given $x \in A_{0}'$, let  $\Omega(x)$ be the collection of all $\mathcal{E_{Z}}$-pre-images $V$ of $A_{0}'$ such that $x \in V$.

Let $h(x) = \{f^{n}(x) \mid n \,\, \text{is a zooming time of} \,\, x\}$. The set $\Omega(x)$ is not empty for every $x \in A_{0}'$ that has a zooming return to $A_{0}'$. Indeed, if $x \in A_{0}'$ and $f^{n}(x) \in A_{0}' \cap h(x)$, then the ball $B_{\delta}(f^{n}(x)) = f^{n}(V_{n}(x)) \supset A_{0}'$ (because $\text{diameter}(A_{0}') < \frac{\delta}{2}$). Thus, for each $h$-return of a point $x \in A_{0}'$ we can associate the $\mathcal{E_{Z}}$ - pre-image $P = f_{|V_{n}(x)}^{-n}(A_{0}')$ of $A_{0}'$ with $x \in P$. 

\begin{definition}
The inducing time on $A_{0}'$ associated to "the first $\mathcal{E_{Z}}$-return to  $A_{0}'$" is the function $R: B_{r}(p) \to \mathbb{N}$ given by
\[
R(x) = 
\left\{
\begin{array}{rcl}
\min\{ord(V) \mid V \in \Omega(x)\}, \mbox{if} \,\, \Omega(x) \neq \emptyset\\
0\,\,\,\,\,\,\,\,\,\,\,\,\,\,\,\,\,\,\,\,\,\,\,\,\,\,\,\, , \mbox{if} \,\, \Omega(x) = \emptyset.
\end{array}
\right.
\] 
\end{definition} 

\begin{definition}
The induced map F associated to "the first $\mathcal{E_{Z}}$-return to $A_{0}'$" is the map given by $F(x) = f^{R(x)}(x), \,\, \text{for all} \, x \in A_{0}'$.
\end{definition}

Since the collection $\Omega(x)$ is totally ordered by inclusion it follows from the Corollary \ref{Main} that there is a unique $I(x) \in \Omega(x)$ such that $ord(I(x)) = R(x)$ whenever $\Omega(x) \neq \emptyset$.

\begin{lemma}
If $\Omega(x) \neq \emptyset \neq \Omega(y)$ then either $I(x) \cap I(y) = \emptyset$ or $I(x) = I(y)$
\end{lemma}

\begin{definition}
The Markov partition associated to "the first $\mathcal{E_{Z}}$-return to $A_{0}'$" is the collection of open sets $\mathcal{P}$ given by $\mathcal{P} = \{I(x) \mid x \in A_{0}' \,\, and \,\,\Omega(x) \neq \emptyset\}$.
\end{definition}

The following corollary guarantees that $P$ is indeed a Markov partition of open sets.

\begin{corollary}
Let $F,R,\mathcal{P}$ be as above. If $\mathcal{P} \neq \emptyset$, then $(F,\mathcal{P})$ is an induced full Markov map for $f$ on $A_{0}'$.
\end{corollary}

We have an induced full Markov map defined on $A_{0}'$. It remains to show that the induced full Markov map defined on $A_{0}'$, is adapted the hole $H$. In fact, given an element $P \in \mathcal{P}$ and $0 \leq j < R(P)$ and suppose that we have $f^{j}(P) = f^{j - R(P)} (A_{0}') \cap H \neq \emptyset$. There is $i \in \{1,2,\dots,k\}$ such that $f^{j}(P) = f^{j - R(P)} (A_{0}') \cap A_{i}' \neq \emptyset$. As the collection $\mathcal{A}'$ is nested, we must have $f^{j}(P) \subset A_{i}' \subset H$. This holds for all $P \in \mathcal{P}$ and $0 \leq j < R(P)$. Then, the induced full Markov map is adapted to the hole $H$ as we claimed.

In order to show that the collection of inducing schemes adapted to the hole $H$ is a Markov structure, we use the proof of Theorem D in \cite{Pi1}, which shows that a collection of inducing schemes, as we have constructed here, has the property that every ergodic probability is lifitable to some of them.

\begin{remark}
We observe that our construction differs from the one in \cite{Pi1} in two aspects. Firstly, Pinheiro considered nested collections with only one element. Also, there the dynamics is not open. But everything is made here along the same lines. Since the result is similar, almost all consequences hold here.
\end{remark}

\begin{remark}
Considering $\Lambda$ as a zooming set (that is, any previously fixed forward invariant set with positive frequency of zooming times and full measure with respect to a zooming measure of reference), we will treat the case where $H \cap \Lambda = \emptyset$, which means that every invariant zooming measure give full measure to the survivor set. The case $H = \emptyset$ will be considered when the zooming set is dense and reducing to closed maps.
\end{remark}

\section{Equilibrium states and conformal measures for the lifted dynamics}\label{Lifted}

In this section we begin the proof of Theorem $\ref{B}$ and the first part of Theorem \ref{C} of the existence of equilibrium states. The strategy is to lift the dynamics to the Markov Structure obtained in Theorem \ref{A}, finding equilibrium states for the induced potential and then projecting them. For the pseudo-conformal measure, later on we will spread it by using a modified inducing scheme.

\subsection{Markov shifts}

Now we recall the basic definitions of symbolic dynamics. Given a countable set $S$, we define the \textbf{\textit{space of symbols}} 
\[
\Sigma  = \{(x_{1},x_{2},\dots,x_{n},\dots) |x_{i} \in S, \text{for all} \, i \in \mathbb{N}\}.
\]
The \textbf{\textit{shift map}} $\sigma : \Sigma \to \Sigma$
is defined by
\[
\sigma(x_{1},x_{2},\dots,x_{n},\dots) = (x_{2},x_{3},\dots,x_{n},\dots).
\]
A \textbf{\textit{cylinder}} is a set of the form 
\[
C_{n} = \{x \in \Sigma \mid  x_{1} = a_{1}, \dots, x_{n} = a_{n}\}.
\]
When an inducing scheme $(F,\mathcal{P})$ is given, we can define a space of symbols by the following rule. Let $x \in \mathcal{U}$ be a point such that 
$F^{k}(x)$ is well defined for all $k \in \mathbb{N}$. To obtain a sequence $(x_{1},x_{2},\dots,x_{n},\dots)$, we put $x_{i} = j$ if $F^{i}(x) \in P_{j}$. 
So, we can see that the map $F$ is conjugate to the shift map. The advantage here is that we can use the theory of symbolic dynamics to obtain results 
for our original map.

\subsection{Locally H\"older potentials}

The results in this section are true for potentials $\phi : M \to \mathbb{R}$ such that the induced potential $\overline{\phi}$ is locally Hölder (see definition below). We will use a mild condition on the contractions $(\alpha_{n})_{n}$ to show that the induced potential of a Hölder potential is locally Hölder. 

Given a potential $\phi : M \to \mathbb{R}$ and an inducing scheme $(F,\mathcal{P})$, we define the {induced potential} as 
\begin{equation}\label{indpot}
	\overline{\phi}(x) =\sum_{j=0}^{R(x)-1} \phi(f^{j}(x)).
\end{equation}
Given a potential $\Phi : \Sigma \to \mathbb{R}$, we say that $\Phi$ is \emph{locally H\"older} if there exist $A > 0$ and 
$\theta \in (0,1)$ such that for all $n \in \mathbb{N}$ 
\[
V_{n}(\Phi) : = \sup\left\{|\Phi(x) - \Phi(y)| \colon  x, y \in C_{n}\right\} \leq A\theta^{n}.
\]
We say that a potential $\Phi : \Sigma \to \mathbb{R}$ has \emph{summable variation} if $\sum_{n=1}^{\infty}V_{n}(\Phi) < \infty$. It is easy to see that every locally Hölder potential has summable variation. 

\begin{lemma}
If the contraction $(\alpha_{n})_{n}$ satisfies $\alpha_{n}(r) \leq ar$ for some $a \in (0,1)$, every $n \in \mathbb{N}$ and every $r \in [0,+\infty)$ and $\phi:M \to \mathbb{R}$ is H\"older, the induced potential $\overline{\phi}$ is locally H\"older.
\end{lemma}
\begin{proof}
	As $\phi$ is H\"older, there are constants $\rho, \alpha > 0$ such that $| \phi(x) - \phi(y) | \leq \rho d(x,y)^{\alpha}$. We must show that there are   
	$A > 0$ and $\theta \in (0,1)$ such that 
	\begin{equation}\label{goal}
		|\overline{\phi}(x) - \overline{\phi}(y)| \leq A \theta^{n}, \text{for all} \, \, \, x,y \in C_{n}.
	\end{equation}
	In fact, given $x,y \in C_{n}$, 
	there are $P_{i_{0}}, P_{i_{1}}, \dots, P_{i_{n}}$ such that $F^{k}(x), F^{k}(y) \in P_{i_{k}}$. Then, we have
	\begin{eqnarray*}
		|\overline{\phi}(x) - \overline{\phi}(y) | &= &
		\left|
		\sum_{j=0}^{R_{i_{0}}-1}
		\phi(f^{j}(x)) - \sum_{j=0}^{R_{i_{0}}-1}\phi(f^{j}(y))\right|
		\leq 
		\sum_{j=0}^{R_{i_{0}}-1}
		\left|\phi(f^{j}(x)) - \phi(f^{j}(y))\right|\\
		&\leq &
		\rho \sum_{j=0}^{R_{i_{0}} - 1} d(f^{j}(x),f^{j}(y))^{\alpha}
		\leq  \rho \sum_{j=0}^{R_{i_{0}} - 1} (\alpha_{j}(d(F(x),F(y)))^{\alpha}\\
		&\leq &
		\rho \sum_{j=0}^{R_{i_{0}} - 1} (\alpha_{R_{i_{1}}}(\alpha_{j}(d(F^{2}(x),F^{2}(y))))^{\alpha}\\	
		&\leq& 
		\rho \sum_{j=0}^{R_{i_{0}} - 1} (\alpha_{R_{i_{n}}}(\dots(\alpha_{R_{i_{1}}}(\alpha_{j}(d(F^{n}(x),F^{n}(y))))))^{\alpha}\\ 
		&\leq &
		\rho \sum_{j=1}^{R_{i_{0}}} (a^{n} \alpha_{j}(d(F^{n}(x),F^{n}(y))^{\alpha} 
		\leq 
		\rho (a^{\alpha})^{n}\sum_{j=1}^{R_{i_{0}}} \alpha_{j}(2\delta)^{\alpha}\\
		&\leq &
		\rho \sum_{j=1}^{\infty} \alpha_{j}(2\delta)^{\alpha} (a^{\alpha})^{n}.\\  
	\end{eqnarray*}
	This yields~\eqref{goal}  with $A  = \rho \sum_{j=1}^{\infty} \alpha_{j}(2\delta)^{\alpha}$ and $\theta  = a^{\alpha}$, that is finite because the contraction is summable. Then, the induced potential is locally H\"older.
\end{proof}
\begin{remark}
	We remind that if the contraction is Lipschitz and, in particular, exponential, the hypothesis is satisfied because for $\alpha_{n}(r) = a_{n}r$ we have $a_{n} \to 0$ and we can take $a = \max\{a_{n}\} \in (0,1)$.
\end{remark}

\subsection{Equilibrium states and conformal measures for Markov shifts}

It follows from  Theorem $\ref{A}$ that there exist an inducing scheme $(F,\mathcal{P})$ and a sequence~$\mu_{n}$ of liftable ergodic probabilities such that 
 $h_{\mu_{n}}(f) + \int \phi d\mu_{n} \to P_{f}(\phi)$ in the case when $\phi$ is zooming. The next result establishes \emph{Abramov's Formulas}.

\begin{proposition}[Zweim\" uller \cite{Z}]\label{pr.zwei} If $\mu$  is liftable to $\overline{\mu}$, then  
\[
h_{\mu}(f) = \frac{h_{\overline{\mu}}(F)}{\int R d \overline{\mu}} \quad\text{and}\quad \int \phi d \mu = \frac{\int \overline{\phi} d\overline{\mu}}{\int R d \overline{\mu}}.
\]
\end{proposition}
It follows from Proposition~\ref{pr.zwei} that
\[
h_{\overline{\mu}}(f) + \int \overline{\phi} d\overline{\mu} = \bigg{(}\int R d \overline{\mu}\bigg{)} \bigg{(}h_{\mu}(f) + \int \phi d\mu\bigg{)}.
\]
Given a Markov shift $(\Sigma,\sigma)$, we define the \emph{Gurevich Pressure} as
\[
\displaystyle P_{G}(\overline{\phi})= \lim_{n \to \infty} \frac{1}{n} \log \left(\sum_{\sigma^{n}(x) = x, x_{0} = a} e^{\overline{\phi}_{n}(x)} \right),
\]
where $\overline{\phi}_{n}(x) = \sum_{j=0}^{n - 1} \overline{\phi}(f^{j}(x)).$

\begin{theorem}[Sarig \cite{Sa1}]\label{Sarig} If $(\Sigma, \sigma)$ is topologically mixing and $\overline{\phi}$ is locally H\"older, then the Gurevich Pressure is well defined and
independent of $a$. 
Moreover,
\[
P_{G}(\overline{\phi}) = \sup \left\{ P_{\topo}(\overline{\phi}_{|Y}) \mid Y \subset \Sigma \,\, \text{is a topologically mixing finite Markov shift}\right\}.
\]
\end{theorem}

\begin{theorem}[Iommi-Jordan-Todd \text{\cite[Theorem 2.1]{IJT}}] \label{VarPrinc} 
	Let $(\Sigma,\sigma)$ be a countable full shift and 
	$\Phi : \Sigma \to \mathbb{R}$ a potential with summable variation and $\sup \Phi < \infty$. Then
	$$
	P_{G}(\Phi) = \sup \bigg{\{}h_{\nu}(F) + \int \Phi d \nu \mid -\int\Phi d\nu < \infty\bigg{\}}.
	$$
\end{theorem}

For the potential $\varphi  = \phi - P_{f}(\phi)$ we have $P_{f}(\varphi) = 0$. 

\begin{proposition}\label{FiniteSup}
	If  $\varphi$ is a potential such that $\overline{\varphi}$ is locally H\"older  and $P_{f}(\varphi)=0$, then   $\sup \overline{\varphi}< \infty$.
\end{proposition}

\begin{proof}
	Since $P_{f}(\varphi)=0$, the Variational Principle for compact sets (see \cite{OV},\cite{W}) implies that
	\[
	0 = \sup_{\eta \in \mathcal{M}_{f}(M)} \bigg{\{} h_{\eta}(f) + \int \varphi d\eta \bigg{\}}.
	\]
	Once $h_{\eta}(f) \geq 0$ for every probability $\eta \in \mathcal{M}_{f}(M)$, it means that $\int \varphi d\eta \leq 0$ for every probability $\eta \in \mathcal{M}_{f}(M)$. The Abramov's Formulas imply that $\int \overline{\varphi} d\overline{\eta} \leq 0$ for every probability $\eta \in \mathcal{M}_{f}(M)$. 
	
	We claim that in the inducing scheme $(F,\mathcal{P})$ that we are considering, the partition $\mathcal{P} = \{P_{1},\dots,P_{n},\dots\}$ is such that each element $P_{i}$ has a point $x_{i}$ such that $\overline{\varphi}(x_{i}) \leq 0$. In fact, given $P_{i} \in \mathcal{P}$, there exists $x_{i} \in P_{i}$ such that $F(x_{i}) = x_{i}$. It means that $x_{i} \in M$ is a periodic point for the map $f$ with period $R(x_{i})$. For the probability
	\[
	\eta_{i} := \frac{1}{R(x_{i})}\sum_{j=0}^{R(x_{i}) - 1}\delta_{f^{j}(x_{i})}
	\]
	we have that $\overline{\eta}_{i} = \frac{1}{R(x_{i})}\delta_{x_{i}}$ and $\int \overline{\varphi} d\overline{\eta}_{i} \leq 0$ implies that $\overline{\varphi}(x_{i}) \leq 0$.
	Since $\overline{\varphi}$ is locally H\"older, which implies that
	\[
	V_{1}(\overline{\varphi}) : = \sup\left\{|\overline{\varphi}(x) - \overline{\varphi}(y)| \colon  x, y \in P_{i}\right\} \leq A\theta.
	\]
	But we obtain $\overline{\varphi}(x)\leq \overline{\varphi}(x) -\overline{\varphi}(x_{i}) \leq A\theta$, for every $x \in P_{i}$, for every $i \in \mathbb{N}$. It means that $\sup \overline{\varphi} \leq A\theta$.
\end{proof}

\begin{proposition}
If  $\varphi$ is a potential as in Proposition \ref{FiniteSup} such that $P_{f}(\varphi)=0$, then   $P_{G}(\overline{\varphi}) = 0$.
\end{proposition}

\begin{proof}
See proof in \cite{AOS}[Proposition 4]. It is the same in our context by using Theorem \ref{A} instead of Theorem \ref{Pinheiro}, Theorem \ref{VarPrinc} here instead of Theorem 3 in \cite{AOS} and Proposition \ref{FiniteSup} to give the upper bound for $\overline{\varphi}$. \\
\end{proof}

Let $T=(t_{ij})$ be the transition matrix of the shift $(\Sigma,\sigma)$. We say that it  has the \emph{big images and preimages (BIP) property} if 
\[
\exists \,\, b_{1},\dots,b_{N} \in S \,\, \text{such that} \,\, \text{for all} \, \, \, a \in S \,\, \exists \,\, i,j \in \{1,\dots,N\} \,\, \text{such that} \,\, t_{b_{i}a}t_{ab_{j}} = 1.
\]
Clearly, if $(\Sigma,\sigma)$ is a full shift, then it has the BIP property.
We are going    to use   Sarig's results on the existence and uniqueness of conformal measures, Gibbs measures and equilibrium states for countable Markov shifts  \cite{Sa1,Sa2,Sa3}; see also \cite{Sa4}. They can   be summarized as follows.

\begin{theorem}[Sarig \cite{Sa4}]
Let $(\Sigma,\sigma)$ be topologically mixing and $\Phi$ have  summable variation. Then $\Phi$ has an invariant Gibbs measure 
$\mu_{\Phi}$ if, and only if, it has the BIP property and $P_{G}(\Phi) < \infty$. Moreover, the Gibbs measure has the following properties:

\begin{enumerate}

\item If $h_{\mu_{\Phi}} (\sigma) < \infty$ or $-\int \Phi d \mu_{\Phi} < \infty$,
then $\mu_{\Phi}$ is the unique equilibrium state
(in particular, $P_{G}(\Phi) = h_{\mu_{\Phi}}(\sigma) + \int \Phi d\mu_{\Phi}$).

\item $\mu_{\Phi}$ is finite and $d \mu_{\Phi} = h_{\Phi} d m_{\Phi}$, where $L(h_{\Phi}) =
\lambda h_{\overline{\varphi }}$ and $L^{*}(m_{\Phi}) = \lambda m_{\Phi}$, $\lambda = e^{P_{G}(\Phi)}$. 
This means that $m(\sigma(A)) = \int_{A} e^{\Phi - \log \lambda} dm_{\Phi}$.

\item  $h_{\Phi}$ is unique and $m_{\Phi}$ is the unique $(\Phi - \log \lambda)$-conformal probability measure.
\end{enumerate}

\end{theorem}
Here $L_{\Phi}$ denotes the Transfer Operator:
\[
L_{\Phi}h = \sum_{y \in \sigma^{-1}(x)} e^{\Phi(y)}h(y).
\]

We obtain a unique equilibrium state 
$\mu_{\overline{\varphi}}$, which is also a Gibbs measure, where the potential $\varphi$ is taken as in Proposition \ref{FiniteSup}.
We  need to show that the inducing time is integrable in order to project this Gibbs measure. Later, we will spread the conformal measure for the original map. Since $P_{G}(\overline{\varphi}) = 0$, we have that $\log \lambda = 0$ and $m_{\overline{\varphi}}$ is a $\overline{\varphi}$-conformal probability measure. See definition of $\overline{\varphi}$ in \ref{indpot}.

\section{Finiteness of ergodic equilibrium states}\label{Finiteness}

In this section, we briefly state the results that can be found in [\cite{AOS}, section 3], where the authors extend the technique of Iommi and Todd in \cite{IT2} to prove that the inducing time is integrable with respect to the Gibbs measure $\mu_{\overline{\varphi}} $.
As a consequence, we can project it to a measure $\mu_{\varphi}$, which will be   an equilibrium state for the original system $(f,\varphi)$. It can be done for each inducing scheme of the Markov structure. So, at the end we obtain finitely many equilibrium states.

We will state results  adapted from \cite{IT2}  whose proofs follow exactly along the same lines.

The strategy is, firstly, to show that for a certain inducing scheme, a measure with low free energy is far from being a Gibbs measure \textit{(Proposition \ref{Gibbs})}. As a consequence, the accumulation point, in the weak-* topology, of a sequence of measures with free energy converging to zero (pressure) will provide us a Gibbs measure, which is $\mu_{\overline{\varphi}}$, by uniqueness. This measure is an equilibrium state and, by the Abramov's formulas, can be projected to an equilibrium state $\mu_{\varphi}$ \textit{(Proposition \ref{Finitely})}.

\subsection{Measures with low free energy}

The following proposition shows that, when we choose a suitable $k \in \mathbb{N}$, the inducing scheme for $F^{k}$ is such that the following property holds: measures with low free energy cannot be Gibbs measures. It can be seen by comparing it with $\mu_{\overline{\varphi}}$.

\begin{proposition}
\label{Gibbs}
Given a potential $\varphi$ with $P_{f}(\varphi)=0$ and an inducing scheme $(\tilde{F},\mathcal{P})$, there exists $k \in \mathbb{N}$ such that replacing $(\tilde{F},\mathcal{P})$ by $(F,\mathcal{P})$, where $F=\tilde{F}^{k}$, the following holds: there exists $\gamma_{0} \in (0,1)$ and for any cylinder $C_{n} \in \mathcal{P}_{n}^{F}$ and $n \in \mathbb{N}$ there exists   $\delta_{n} < 0$ such that any measure $\mu_{F} \in \mathcal{M}_{F}$ with
\[
\mu_{F}(C_{n}) \leq (1 - \gamma_{0}) m_{\overline{\varphi}}(C_{n}) \quad\text{or}\quad m_{\overline{\varphi}}(C_{n}) \leq (1 - \gamma_{0}) \mu_{F}(C_{n}), 
\]
where $m_{\overline{\varphi}}$ is the conformal measure for  $(F,\overline{\varphi})$, 
satisfies $h_{\mu_{F}}(F) + \int \overline{\varphi} d \mu_{F} < \delta_{n}$.
\end{proposition}
\begin{proof}
See proof in \cite{AOS} or \cite{IT2}.
\end{proof}

\subsection{Integrability of the inducing time}

To finish the proof of Theorem A we will use the following proposition.

\begin{proposition}
\label{Finitely}
Given a pseudo-geometric potential $\varphi$ with $P_{f}(\varphi)=0$, there exist finitely many ergodic equilibrium states for the system $(f,\varphi)$ and they are zooming.
\end{proposition}

In order to prove the above proposition, we take a sequence of $f$-invariant measures $\{\mu_{n}\}_{n}$ such that $h_{\mu_{n}}(f) + \int \varphi d \mu_{n} \to 0$ and liftable 
with respect to the same inducing scheme. We will show that the set of lifted measure is tight (see definition below and \cite{OV} for details)
and has the Gibbs measure as its unique accumulation point with respect to which the inducing time is integrable.
Finally, the Gibbs measure is projected to an equilibrium state for $(f,\varphi)$.

We say that a set of measures $\mathcal{K}$ on $X$ is \textbf{\textit{tight}} if for every $\epsilon > 0$ there exists a compact set $K \subset X$ 
such that $\eta(K^{c}) < \epsilon$ for every measure $\eta \in \mathcal{K}$.
\begin{lemma}\label{Tight}
The lifted sequence $\{\overline{\mu}_{n}\}_{n}$  is tight.
\end{lemma}

\begin{proof}
See proof in \cite{AOS}[Lemma 6]. It is the same in our context by using Theorem \ref{A} instead of Theorem \ref{Pinheiro}.
\end{proof} 

\begin{proof}[Proof of Proposition~\ref{Finitely}]
See proof in \cite{AOS}.
\end{proof}

As there exists an equilibrium state, we also can find an ergodic one. Also, if $\nu$ is an ergodic equilibrium state, we can lift it to an equilibrium state for the shift, which is the Gibbs measure $\mu_{\overline{\varphi}}$. So, the projection of it is $\nu$. It shows that there exists at most one ergodic equilibrium state for each inducing scheme. Then, they are, at most, finitely many. 

\section{Uniqueness of Equilibrium State}\label{Uniqueness}

In the previous section we obtained finiteness of equilibrium states and we will show that we can obtain uniqueness under the conditions of Theorem \ref{B}.

\begin{theorem}\label{uniqueness}
	Let us assume that there exist infinitely many ergodic zooming measures whose supports are pairwise disjoints. For every pair of ergodic zooming equilibrium states $\mu_{1}$ and $\mu_{2}$ for the zooming potential $\phi$ whose induced potential is locally Hölder, we show that they are liftable to the same inducing scheme and we obtain uniqueness of equilibrium state with no requirement of transitivity.
\end{theorem}
\begin{proof}
	We already have finiteness of equilibrium states. The Markov structure provides at most one equilibrium state for each inducing scheme. We are going to show that the equilbrium states are liftable to the same inducing schemes in the Markov structure, under the given hypothesis.
	
	The construction of the Markov structure is made by taking a  cover of the space  by open balls with radius $r>0$ and a finite subcover $\{B_{r}(p_{1}),\dots, B_{r}(p_{k})\}$ which provides another subcover by  nested sets $B_{r}^{*}(p_{i}) \subset B_{r}(p_{i}) $ which contains $B_{r/2}(p_{i})$ with $B_{r}(p_{i}) \cap B_{r/2}(p_{j}) = \emptyset, i \neq j, i,j \leq k$.
	
	Let $\mu_{1}, \mu_{2}$ be ergodic equilibrium states for the zooming potential $\phi$. The equilibrium states are zooming measures: $\mu_{1}(\Lambda) = 1 = \mu_{2}(\Lambda)$, where $\Lambda$ is the zooming set. By \cite{Pi1}[Lemma 3.9] there exist compact sets $A_{\mu_{1}}$ and $A_{\mu_{2}}$ such that $\omega_{f,+}(x) = A_{\mu_{1}}$ $\mu_{1}$-almost every $x \in M$ and $\omega_{f,+}(y) = A_{\mu_{2}}$ $\mu_{2}$-almost every $y \in M$. Set 
	\[
	B_{\mu_{1}} = \{w \in M \mid \omega_{f,+}(w) = A_{\mu_{1}}\}, B_{\mu_{2}} = \{w \in M \mid \omega_{f,+}(w) = A_{\mu_{2}}\} \implies \mu_{1}(B_{\mu_{1}}) = 1 = \mu_{2}(B_{\mu_{2}}),
	\]
	Assume that $\mu_{1} \neq \mu_{2}$ and $B_{\mu_{1}} \cap B_{\mu_{2}} = \emptyset$.  Once there exist finitely many ergodic equilibrium states, we can take an ergodic zooming measure $\nu$ which is not an ergodic equilibrium state. We can also consider the set
	\[
	B_{\nu} = \{w \in M \mid \omega_{f,+}(w) = A_{\nu}\} \implies \nu(B_{\nu}) = 1
	\]
	and denoting $\mu : = \mu_{1}$ define the function
	\[
	\varphi = \phi + \chi_{(B_{\mu} \cup B_{\nu})}.
	\]
	Denoting by $\chi_{A}$ the characteristic function of the measurable set $A$, we have
	\[
	\int \varphi \cdot \chi_{A} d\mu =  \int \phi \cdot \chi_{A} d\mu + \mu(A), \int \varphi \cdot \chi_{A} d\nu =  \int \phi \cdot \chi_{A} d\nu + \nu(A) \implies
	\]
	\[
	\int_{A} (\varphi - \phi) d\mu = \mu(A),  \int_{A} (\varphi - \phi) d\nu = \nu(A).
	\]
	Taking $A = B_{\mu} \cup B_{\nu}$, we have $\mu(A) = 1 = \nu(A)$ and for $\nu$-almost every $y \in M$ we obtain
	\[
	(*) \lim_{n \to \infty} \frac{1}{n}\sum_{i=0}^{n-1}(\varphi - \phi)\cdot \chi_{A}(f^{i}(y)) = \int_{A} (\varphi - \phi) d\nu = \int_{A} (\varphi - \phi) d\mu =
	\]
	\[
	\int (\varphi - \phi) d\mu = \int (\varphi - \phi) d\nu = \lim_{n \to \infty} \frac{1}{n}\sum_{i=0}^{n-1}(\varphi - \phi)(f^{i}(y)) = 1.
	\]
	Taking $y \in M\backslash(B_{\mu}\cup B_{\nu})$, it implies that $(\varphi - \phi)(f^{i}(y)) = 0$ for every $i \in \mathbb{N}$ and then the equality $(*)$ holds for every $y \in B_{\mu}\cup B_{\nu}$ and not for $y \in M\backslash(B_{\mu}\cup B_{\nu})$.
	
	Analogously, taking $A = \Lambda$ and $y \in M\backslash \Lambda$, it implies that $(\varphi - \phi)(f^{i}(y)) = 0$ for every $i \in \mathbb{N}$ and then the equality $(*)$ holds for $\nu$ -almost every $y \in \Lambda$ and not for $y \in M\backslash \Lambda$. We obtain $\Lambda = \overline{B_{\mu}} \cup \overline{B_{\nu}}$ in the induced topology on $\Lambda$ because $\Lambda$ also has $\nu$-full measure.
	
	 We can repeat the idea for $\mu: = \mu_{2}$ and every ergodic zooming measure $\nu$ which is not an equilibrium state, obtaining $\Lambda =  \overline{B_{\mu_{1}}} \cup \overline{B_{\nu}} =  \overline{B_{\mu_{2}}} \cup \overline{B_{\nu}}$. Supposing $ \overline{B_{\mu_{1}}} \cap \overline{B_{\mu_{2}}} = \emptyset$, we conclude that $\Lambda = \overline{B_{\nu}}$ for every ergodic zooming measure $\nu$ which is not an equilibrium state. It is a contradiction because $A_{\nu} \subset \text{supp} \nu$ for every ergodic measure $\nu$. Hence, $ \overline{B_{\mu_{1}}} \cap \overline{B_{\mu_{2}}} \neq \emptyset$. As $B_{\mu_{1}} \cap B_{\mu_{2}} = \emptyset$ we get the compact invariant set $\partial B_{\mu_{1}} \cap \partial B_{\mu_{2}}$ supporting an ergodic zooming measure $\eta$ and containing $A_{\eta}$. We also have points $w_{1} \in B_{\mu_{1}}$ and  $w_{2} \in B_{\mu_{2}}$ such that $\omega_{f,+}(w_{1}) = A_{\mu_{1}} = A_{\eta} = A_{\mu_{2}} = \omega_{f,+}(w_{2})$. It is true because the restriction of a map to the support of an ergodic measure is transitive. Then, some orbit in $B_{\mu_{1}}$ and $B_{\mu_{2}}$ accumulates on $A_{\eta}$ at zooming times.
	
	Once $A_{\mu_{1}} = A_{\mu_{2}}$, as a consequence both measures $\mu_{1}$ and $\mu_{2}$ are liftable to the same inducing scheme. As we can see in the proof of \cite{Pi1}[Theorem D] the measures are liftable to every inducing scheme which intersects this compact set. It implies that $\mu_{1} = \mu_{2}$. Hence, the uniqueness is established.
\end{proof}

\section{Hyperbolic Potentials}\label{Hyperbolic}

In this section, we extend the notion of hyperbolic potentials which appears in \cite{AOS} for exponential contractions and for continuous maps to the general case of $f:M \to M$ being a zooming map and  the contraction $(\alpha_{n})_{n}$ satisfying $\alpha_{n}(r) \leq ar$ for some $a \in (0,1)$, every $n \in \mathbb{N}$ and every $r \in [0,+\infty)$ (Lipschitz contractions, for example). Also, we show that hyperbolic potentials are \textbf{equivalent} to continuous zooming potentials and give an example of a class of hyperbolic potentials. By proving that the null potential is zooming we obtain, in particular, the existence and uniqueness of measures of maximal entropy for zooming maps with general contractions, answering this question for the important class of maps known as Viana maps.

We stress the fact that in the case of a zooming map with non dense zooming set we can consider this section for the respective open zooming system.

We begin by recalling what we mean by relative pressure, which is a notion of pressure for non-compact sets, as the case of zooming sets in general.

\subsection{Topological pressure}

We recall the definition of relative pressure for non-compact sets by dynamical balls, as it is given in \cite{ARS}.
Let $M$ be a compact metric space. Consider $f:M \to M$ and $\phi: M \to \mathbb{R}$. Given $\delta > 0$, $n \in \mathbb{N}$ and $x \in M$, we define the 
\textbf{\textit{dynamical ball}} $B_{\delta}(x,n)$ as 
\[
B_{\delta}(x,n): = \{y \in M |d(f^{i}(x),f^{i}(y)) < \delta, \,\, \text{for} \,\, 0 \leq i \leq n\}.
\]
Consider for each $N \in \mathbb{N}$, the set
\[
\mathcal{F}_{N} = \left\{B_{\delta}(x,n) |x \in M, n \geq N\right\}.
\]
Given $\Lambda \subset M$, denote by $\mathcal{F}_{N}(\Lambda)$ the finite or countable families of elements in $\mathcal{F}_{N}$ that cover $\Lambda$. Define for $n \in \mathbb{N}$
\[
S_{n}\phi(x) = \phi(x) + \phi(f(x)) + \dots + \phi(f^{n-1}(x)).
\]
and
\[
\displaystyle R_{n,\delta}\phi(x) = \sup_{y \in B_{\delta}(x,n)} S_{n}\phi(y).
\]
Given a $f$-invariant set $\Lambda \subset M$, not necessarily compact, define for each $\gamma > 0$
\[
\displaystyle m_{f}(\phi, \Lambda, \delta, N, \gamma) = \inf_{\mathcal{U} \in \mathcal{F}_{N}(\Lambda)} \left\{ \sum_{B_{\delta}(y,n) \in \mathcal{U}}
e^{-\gamma n + R_{n,\delta}\phi(y)} \right\}.
\]
Define also
\[
\displaystyle m_{f}(\phi, \Lambda, \delta, \gamma) = \lim_{N \to + \infty} m_{f}(\phi, \Lambda, \delta, N, \gamma).
\]
and
\[
P_{f}(\phi, \Lambda, \delta) = \inf \{\gamma > 0 |m_{f}(\phi, \Lambda, \delta, \gamma) = 0\}.
\]
Finally, define the \textbf{\textit{relative pressure}} of $\phi$ on $\Lambda$ as
\[
P_{f}(\phi,\Lambda) = \lim_{\delta \to 0} P_{f}(\phi, \Lambda, \delta).
\]
The \textbf{\textit{topological pressure}} of $\phi$ is, by definition, $P_{f}(\phi) = P_{f}(\phi, M)$ and satisfies
\begin{eqnarray}
	\label{Pressures}
	P_{f}(\phi) = \sup \{P_{f}(\phi,\Lambda), P_{f}(\phi,\Lambda^{c})\}
\end{eqnarray}
where $\Lambda^{c}$ denotes the complement of $\Lambda$ on $M$. We refer the reader to \cite{Pe2} for the proof of~\eqref{Pressures} and for additional properties of the pressure.
See also \cite{W} for a proof of the fact that
\[
\displaystyle P_{f}(\phi) = \sup_{\mu \in \mathcal{M}_{f}(M)} \bigg{\{} h_{\mu}(f) + \int \phi d \mu \bigg{\}}.
\]

\subsection{Hyperbolic potentials}

Given a continuous zooming map $f:M \to M$ with general contractions, we say that a   continuous function $\phi : M \to \mathbb{R}$ is a \emph{hyperbolic potential} if the topological pressure $P_{f}(\phi)$ is located on the zooming set $\Lambda$, i.e.
\[
P_{f}(\phi,\Lambda^{c}) < P_{f}(\phi).
\]
This notion is extends the notion of hyperbolic potential in \cite{RV}, since they define the expanding set from an average and, later on, they prove the property of expansion with neighbourhoods by proving a distortion control in Lemma 3.6 and that the expanding set is a zooming set.

In \cite{IRL} I. Inoquio-Renteria and J. Rivera-Letelier use the term hyperbolic potential for the first time. As in \cite{LRL}, where H. Li and J. Rivera-Letelier consider other type of hyperbolic potentials for one-dimensinal dynamics. In their context, $\phi$ is a hyperbolic potential if
\begin{equation}\label{Li-Rivera}
\displaystyle \sup_{\mu \in \mathcal{M}_{f}(M)} \int \phi d\mu < P_{f}(\phi).	
\end{equation}
We claim that these type of hyperbolic potentials are zooming. When the map is one-dimensional, a measure being zooming (or expanding) means that the Lyapunov exponent is positive. Otherwise, it is negative or zero. By Ruelle's inequality, we obtain that the entropy is negative or zero. In this case, for a measure $\mu$ that is not zooming, we obtain
\[
\displaystyle \sup_{\mu \in \mathcal{Z}(\Lambda)^{c}} \bigg{\{} h_{\mu}(f) + \int \phi d \mu \bigg{\}} \leq \sup_{\mu \in \mathcal{Z}(\Lambda)^{c}}\bigg{\{}\int \phi d \mu \bigg{\}} \leq 
\]
\[
\sup_{\mu \in \mathcal{M}_{f}(M)} \int \phi d\mu < P_{f}(\phi)=\sup_{\mu \in \mathcal{Z}(\Lambda)} \bigg{\{} h_{\mu}(f) + \int \phi d \mu \bigg{\}}.	
\]
It means that this type of potential is zooming (and hyperbolic, as we will see in this section) as defined above and we can use our Theorem \ref{B} to obtain finitely many ergodic equilibrium states which are zooming measures.

Another type of hyperbolic potentials considered, is taking $\phi$ such that
\begin{equation}\label{Przytycki-Rivera}
\sup \phi < P_{f}(\phi).
\end{equation}
We can easily see that condition \ref{Przytycki-Rivera} implies \ref{Li-Rivera}.
We observe that every hyperbolic potential in our context is zooming, as we can see in the following proposition.

\begin{proposition}\label{HZ}
	Let $\phi$ be  a hyperbolic potential. If $\mu$ is an ergodic probability measure such that $h_{\mu}(f) + \int \phi d\mu > P_{f}(\phi,\Lambda^{c})$, then $\mu(\Lambda)=1$.
\end{proposition}

\begin{proof}
	Since $\Lambda$ is an invariant set and $\mu$ is an ergodic probability measure, we have $\mu(\Lambda) = 0$ or $\mu(\Lambda) = 1$. Since the potential $\phi$ is ergodic, we get
	\[
	h_{\mu}(f) + \int \phi d\mu > P_{f}(\phi,\Lambda^{c}) \geq \sup_{\nu(\Lambda^{c})=1}
	\bigg{\{}h_{\nu}(f) + \int_{\Lambda^{c}} \phi d\nu\bigg{\}}.
	\]
	(For the second inequality, see  \cite[Theorem A2.1]{Pe2})
	So, we cannot have $\mu(\Lambda^{c}) = 1$ and we obtain $\mu(\Lambda) = 1$, i.e. $\mu$ is an zooming measure.
	
\end{proof}

The main result of this section is the following theorem, which establishes the equivalence between hyperbolic and zooming potentials.

\begin{theorem}\label{HYPERZOOM}
	Let $f:M \to M$  a continuous zooming map and  the contraction $(\alpha_{n})_{n}$ satisfying $\alpha_{n}(r) \leq ar$ for some $a \in (0,1)$, every $n \in \mathbb{N}$ and every $r \in [0,+\infty)$ (Lipschitz contractions, for example)  and $\phi : M \to \mathbb{R}$ a continuous potential. Then $\phi$ is a hyperbolic potential if, and only if, it is a zooming potential.
\end{theorem}

We divide the proof of Theorem \ref{HYPERZOOM} into some Lemmas. The first Lemma proves that both the sets of hyperbolic and continuous zooming potentials are open in the topology of the supremum norm.

\begin{lemma} \label{OPEN}
	Let $f:M \to M$ a continuous zooming map and the contraction $(\alpha_{n})_{n}$ satisfying $\alpha_{n}(r) \leq ar$ for some $a \in (0,1)$, every $n \in \mathbb{N}$ and every $r \in [0,+\infty)$ (Lipschitz contractions, for example). Denote by $\mathcal{HP}$ the set of hyperbolic potentials and $\mathcal{ZP}$ the set of continuous zooming potentials. We have that both sets $\mathcal{HP}$ and $\mathcal{ZP}$ are open in the topology of the supremum norm $\parallel \cdot \parallel_{\infty}$. 
\end{lemma}
\begin{proof}
	Firstly, we prove that the set $\mathcal{HP}$ is open, by using the following property of the relative pressure given in \cite{Pe2}. For every pair of continuous potentials $\phi, \varphi: M \to \mathbb{R}$, we have that
	\[
	|P(\phi,X) - P(\varphi,X)| \leq \parallel \phi - \varphi \parallel_{\infty}, \,\, X \subset M.
	\]
	So, by assuming that $\phi$ is hyperbolic, we have that $P(\phi,\Lambda^{c}) < P(\phi, \Lambda)$ and, taking $0 < 2\epsilon < P(\phi, \Lambda) - P(\phi,\Lambda^{c})$ and $\varphi$ such that $\parallel \phi - \varphi \parallel_{\infty} < \epsilon$, we have 
	\[
	P(\varphi, \Lambda) - P(\varphi,\Lambda^{c}) = (P(\varphi, \Lambda) - P(\phi,\Lambda)) + (P(\phi,\Lambda) - P(\phi,\Lambda^{c})) + (P(\phi,\Lambda^{c}) - P(\varphi,\Lambda^{c})) >
	\]
	\[
	-\parallel \phi - \varphi \parallel_{\infty} + P(\phi,\Lambda) - P(\phi,\Lambda^{c}) - \parallel \phi - \varphi \parallel_{\infty} > P(\phi,\Lambda) - P(\phi,\Lambda^{c}) - 2\epsilon > 0.
	\]
	It implies that $\varphi$ is hyperbolic and, as a consequence, we have the ball $B_{\epsilon}(\phi) \subset \mathcal{HP}$. It shows that the set $\mathcal{HP}$ is open.
	Now, in order to show that the set $\mathcal{ZP}$ is open, we assume that the potential $\phi$ is zooming, that is,
	\[
	\displaystyle  \sup_{\eta \in \mathcal{Z}(\Lambda)^{c}} \bigg{\{} h_{\eta}(f) + \int \phi d\eta \bigg{\}} < \sup_{\eta \in \mathcal{Z}(\Lambda)} \bigg{\{} h_{\eta}(f) + \int \phi d\eta \bigg{\}}.
	\]
	So, taking
	\[
	\displaystyle 0 < 2\epsilon < \sup_{\eta \in \mathcal{Z}(\Lambda)} \bigg{\{} h_{\eta}(f) + \int \phi d\eta \bigg{\}} - \sup_{\eta \in \mathcal{Z}(\Lambda)^{c}} \bigg{\{} h_{\eta}(f) + \int \phi d\eta \bigg{\}},
	\]
	and $\varphi$ such that $\parallel \phi - \varphi \parallel_{\infty} < \epsilon$ we have
	\[
	\displaystyle  \sup_{\eta \in \mathcal{Z}(\Lambda)} \bigg{\{} h_{\eta}(f) + \int \varphi d\eta \bigg{\}} - \sup_{\eta \in \mathcal{Z}(\Lambda)^{c}} \bigg{\{} h_{\eta}(f) + \int \varphi d\eta \bigg{\}} = \sup_{\eta \in \mathcal{Z}(\Lambda)} \bigg{\{} h_{\eta}(f) + \int \varphi d\eta \bigg{\}} -
	\]
	\[
	- \sup_{\eta \in \mathcal{Z}(\Lambda)} \bigg{\{} h_{\eta}(f) + \int \phi d\eta \bigg{\}} +	\displaystyle  \sup_{\eta \in \mathcal{Z}(\Lambda)} \bigg{\{} h_{\eta}(f) + \int \phi d\eta \bigg{\}} - \sup_{\eta \in \mathcal{Z}(\Lambda)^{c}} \bigg{\{} h_{\eta}(f) + \int \phi d\eta \bigg{\}} +
	\]
	\[
	+ \sup_{\eta \in \mathcal{Z}(\Lambda)^{c}} \bigg{\{} h_{\eta}(f) + \int \phi d\eta \bigg{\}} - \sup_{\eta \in \mathcal{Z}(\Lambda)^{c}} \bigg{\{} h_{\eta}(f) + \int \varphi d\eta \bigg{\}} > 
	\]
	\[
	\displaystyle \sup_{\eta \in \mathcal{Z}(\Lambda)} \bigg{\{} h_{\eta}(f) + \int \phi d\eta \bigg{\}} - \sup_{\eta \in \mathcal{Z}(\Lambda)^{c}} \bigg{\{} h_{\eta}(f) + \int \phi d\eta \bigg{\}} - 2\epsilon > 0.
	\]
	We used that $\parallel \phi - \varphi \parallel_{\infty} < \epsilon$ implies $\int \varphi d\eta - \int \phi d\eta > -\epsilon$ and also $\int \phi d\eta - \int \varphi d\eta > -\epsilon$. We obtained that $\varphi$ is zooming and, as a consequence, we have the ball $B_{\epsilon}(\phi) \subset \mathcal{ZP}$. It shows that the set $\mathcal{ZP}$ is open.
	\end{proof}

We observe that Proposition \ref{HZ} guarantees that $\mathcal{HP} \subset \mathcal{ZP}$. In order to prove Theorem \ref{HYPERZOOM} it remains to show that $\mathcal{ZP} \subset \mathcal{HP}$. In the next Lemma, we show this by showing that $\mathcal{ZP} \subset \overline{\mathcal{HP}}$ and using that $\mathcal{HP} \subset \mathcal{ZP}$ are open sets.

\begin{lemma}\label{ZOOMHYPER}
	With the notation of Lemma \ref{OPEN}, we have that $\mathcal{ZP} \subset \overline{\mathcal{HP}}$. 
\end{lemma}
   \begin{proof}
   	Let $\phi \in \mathcal{ZP}$ be a Hölder potential, which exists because $\mathcal{ZP}$ is open and the set of Hölder functions are dense in the set of continuous functions. By Theorem \ref{B}, there exist finitely many equilibrium states, which are also zooming measures. Let $\mu_{0}$ be an equilibrium state. If $\phi \in \mathcal{ZP} \backslash \mathcal{HP}$, we have the following inequalities
   	\[
   	P(\phi,\Lambda) \leq P(\phi,\Lambda^{c}) = P(\phi) = h_{\mu_{0}} + \int \phi d\mu_{0} \leq \sup_{\eta \in \mathcal{Z}(\Lambda)} \bigg{\{} h_{\eta}(f) + \int \phi d\eta \bigg{\}} \leq P(\phi,\Lambda),  
   	\]
    because $\mu_{0}$ is zooming and the last inequality is a consequence of  \cite[Theorem A2.1]{Pe2}. Hence, we obtained that $P(\phi,\Lambda) = P(\phi,\Lambda^{c}) = P(\phi)$. Once the relative pressure is continuous ($|P(\phi,X) - P(\varphi,X)| \leq \parallel \phi - \varphi \parallel_{\infty}, \,\, X \subset M$), $\overline{\mathcal{HP}}$ is the set of continuous potentials $\varphi$ such that $P(\varphi,\Lambda^{c}) \leq P(\varphi,\Lambda)$. We conclude that $\phi \in \overline{\mathcal{HP}}$. Then, every Hölder zooming potential is in $\overline{\mathcal{HP}}$. By denseness of the Hölder functions, we have that $\mathcal{ZP} \subset \overline{\mathcal{HP}}$, as we wished.
   \end{proof}
 
 By Lemma \ref{ZOOMHYPER}, we have that $\mathcal{ZP} \subset \overline{\mathcal{HP}}$. Since $\mathcal{ZP}$ is an open set, we conclude that $\mathcal{ZP} \subset \mathcal{HP}$ and, finally, $\mathcal{ZP} = \mathcal{HP}$, concluding the proof of Theorem \ref{HYPERZOOM}.

\begin{remark}
	We observe that the set $\mathcal{ZP}$ could be empty, but we have that it contains the null potential $\phi \equiv 0$ as can be seen in Theorem \ref{ZER} next.
\end{remark}

\begin{theorem}\label{ZER}
	Let $f:M \to M$ a continuous zooming map and  the contraction $(\alpha_{n})_{n}$ satisfying $\alpha_{n}(r) \leq ar$ for some $a \in (0,1)$, every $n \in \mathbb{N}$ and every $r \in [0,+\infty)$ (Lipschitz contractions, for example) such that the topological entropy is positive ($h(f) > 0$). Then the null potentials $\phi \equiv 0$ is a zooming potential (and also hyperbolic). In particular, by Theorem \ref{B} there exist finitely many ergodic measures of maximal entropy which are zooming measures and we have uniqueness by Theorem \ref{uniqueness}.
\end{theorem}
\begin{proof}
	In order to prove Theorem \ref{ZER}, we use Theorem \ref{C} on the existence of pseudo-geometric zooming potentials, which has equilibrium states which are zooming measures. 
		
	We recall that $\mathcal{Z}(\Lambda)$ denotes the set of zooming invariant measures. Assume, by contradiction, that
	\[
	\sup_{\mu \in \mathcal{Z}(\Lambda)}\{h_{\mu}(f)\} \leq \sup_{\nu \in \mathcal{Z}(\Lambda)^{c}}\{h_{\nu}(f)\}.
	\]
	Let $\phi_{0}$ be a pseudo-geometric zooming potential given in Theorem \ref{C}. There exists an ergodic equilibrium state $\mu \in \mathcal{Z}(\Lambda)$. We  have that $\mu$ is also an equilibium state for $\phi: = \phi_{0} - P(\phi_{0})$ which implies that $P(\phi) = 0$ and $h_{\mu}(f) + \int \phi d\mu = 0$ and $-h_{\mu}(f) = \int \phi d\mu \leq 0$. Assume that there exists $\nu_{0} \in \mathcal{Z}(\Lambda)^{c}$ such that $h_{\mu}(f) < h_{\nu_{0}}(f)$. Since every equilibrium state is a zooming probability, we have that $\nu_{0}$ is not an equilibrium state for $\phi$ and
	\[
	h_{\nu_{0}}(f) + \int \phi d\nu_{0} < h_{\mu}(f) + \int \phi d\mu = 0,
	\]
	which implies that $\int \phi d\nu_{0} < \int \phi d\mu = -h_{\mu}(f)$ because $h_{\mu}(f) < h_{\nu_{0}}(f)$.  For every $t \in (0,1)$ and putting $\nu = (1-t)\mu + t\nu_{0}$ we still have that $h_{\mu}(f) < h_{\nu}(f)$ (we observe that $h_{\nu}(f) = (1-t)h_{\mu}(f) + th_{\nu_{0}}(f)$, because the entropy is an affine function, see \cite[Proposition 9.6.1]{OV}). Let $\epsilon > 0$ and $t \in (0,1), n \in \mathbb{N}$ are chosen such that such that $0 < h_{\nu}(f) - h_{\mu}(f) -\epsilon/n$ and  
	\[ 
	\int \phi d\mu > \int \phi d\nu > \int \phi d\mu - \epsilon \implies 
	- \epsilon(n-1)/n < h_{\nu}(f) - h_{\mu}(f) -\epsilon < h_{\nu}(f) + \int \phi d\nu < 0,
	\]
	using that $\int \phi d\mu = -h_{\mu}(f)$ . It is possible because once $\mu$ is an equilibrium state we have $h_{\nu_{0}}(f) - h_{\mu}(f) < \int \phi d\mu - \int \phi d\nu_{0}$ and there exist $t \in (0,1), n \in \mathbb{N}$ such that
	\[
	\epsilon/n < h_{\nu}(f) - h_{\mu}(f) = t(h_{\nu_{0}}(f) - h_{\mu}(f)) < t\bigg{(}\int \phi d\mu - \int \phi d\nu_{0}\bigg{)} = \int \phi d\mu - \int \phi d\nu < \epsilon.
	\]
	We obtain a contradiction because $\nu \in \mathcal{Z}(\Lambda)^{c}$ and $\epsilon > 0$ is arbitrary and the potential $\phi$ is zooming. It shows that we have
	\[
	\sup_{\nu \in \mathcal{Z}(\Lambda)^{c}}\{h_{\nu}(f)\} \leq h_{\mu}(f).
	\]
	If $\mu_{0}$ is an equilibrium state which is not a measure of maximal entropy, then we have
	\[
	\sup_{\nu \in \mathcal{Z}(\Lambda)^{c}}\{h_{\nu}(f)\} \leq h_{\mu_{0}}(f) < \sup_{\mu \in \mathcal{Z}(\Lambda)}\{h_{\mu}(f)\} = h(f).
	\]
	Hence, the null potential is zooming. As a consequence of Theorem \ref{HYPERZOOM}, we also have that
	\[
	P(0,\Lambda^{c}) < P(0,\Lambda),
	\]
	and the null potential is hyperbolic.
\end{proof}

\begin{corollary}\label{BIR}
	Let $\phi: M \to \mathbb{N}$ with its Birkhoff sums uniformly bounded, that is, there exists $r > 0$ such that
	\[
	|S_{n}\phi(x)| < r, \text{for all} \, \, n \in \mathbb{N}, \text{for all} \, \, x \in M.
	\]
	Then, $\phi$ is a zooming (and hyperbolic) potential.
\end{corollary}
\begin{proof}
	By Birkhoff's Ergodic Theorem we have that $\int \phi d\eta = 0$ for every $f$-invariant probability $\eta$. By Theorem \ref{ZER} the null potential is zooming. It implies that 
	\[
	\displaystyle  \sup_{\eta \in \mathcal{Z}(\Lambda)^{c}} \bigg{\{} h_{\eta}(f) + \int \phi d\eta \bigg{\}} = \sup_{\eta \in \mathcal{Z}(\Lambda)^{c}} \bigg{\{} h_{\eta}(f) \bigg{\}} < \sup_{\eta \in \mathcal{Z}(\Lambda)} \bigg{\{} h_{\eta}(f) \bigg{\}} = \sup_{\eta \in \mathcal{Z}(\Lambda)} \bigg{\{} h_{\eta}(f) + \int \phi d\eta \bigg{\}}.
	\]  
	So, $\phi$ is a zooming (and hyperbolic) potential.
\end{proof}

\begin{remark}
	The previous lemma proved that the potential $\phi$ is zooming (and hyperbolic) if we have $\int \phi d \eta = 0$ for every invariant measure $\eta$. In particular, if we have the Birkhoff sums uniformly bounded.
\end{remark}

A very important class of zooming maps is the class of Viana maps, defined in  section \ref{Examples}. The problem concerning the existence and uniqueness of the measure of maximal entropy has been studied for several authors. In \cite{ALP} the authors prove that there exist at most countably many of them. In \cite{AOS} the authors prove existence and finiteness. A proof of existence and uniqueness is announced in \cite{PV} with a different approach. We obtain it as a corollary of Theorem \ref{ZER}.
\begin{corollary}
	Let $f:S^{1} \times I \to S^{1} \times I$ be a Viana map. There exists a unique measure of maximal entropy for $f$.
\end{corollary}
\begin{proof}
	We have that every Viana map is a zooming map with exponential contractions with positive topological entropy. The corollary follows by Theorem \ref{ZER} and \cite{Pi1}[Corollary 9.5] which guarantees the existence of infinitely many periodic repellers.
\end{proof}

\begin{remark}
	In \cite[Proposition 12.2]{ALP} the authors establish that for Viana maps we have the following result, among others: if $\mu$ is an $f$-invariant measure such that $h_{\mu}(f) \geq h_{SRB}(f)$ where $SRB$ denotes the unique SRB measure for Viana maps, then the measure $\mu$ is hyperbolic. It implies that
		\[
	\sup_{\nu \in \mathcal{Z}(\Lambda)^{c}}\{h_{\nu}(f)\} \leq \sup_{\mu \in \mathcal{Z}(\Lambda)}\{h_{\mu}(f)\} = h(f).
	\]
	If the inequality is strict, it means that the null potential is zooming. Otherwise, we still can take a sequence of zooming measures $\mu_{n}$ such that $h_{\mu_{n}}(f) \to h(f)$ and the proof to find equilibrium states proceeds analogously. We then find uniqueness of the measure of maximal entropy in any case.
\end{remark}

With Theorem \ref{ZER} and the next lemma, we can see that the constant potentials are all hyperbolic (and zooming).

\begin{lemma} 
	We have that $P_{\Lambda}(\phi + c) = P_{\Lambda}(\phi) + c$, for all potential $\phi$ and constant $c \in \mathbb{R}$. 
\end{lemma}

\begin{proof}
	Let $\psi := \phi + c$. We have  $R_{n,\delta}\psi(x) = R_{n,\delta}\phi(x) + nc$. So,
	\begin{eqnarray*}
	\displaystyle m_{f}(\psi, \Lambda, \delta, N, \gamma) & = & \inf_{\mathcal{U} \in \mathcal{F}_{N}(\Lambda)} \bigg{\{} \sum_{B_{\delta}(y,n) \in \mathcal{U}}
	e^{-\gamma n + R_{n,\delta}\psi(y)} \bigg{\}} \\
	= \inf_{\mathcal{U} \in \mathcal{F}_{N}(\Lambda)} \bigg{\{} \sum_{B_{\delta}(y,n) \in \mathcal{U}}
	e^{-(\gamma - c)n + R_{n,\delta}\phi(y)} \bigg{\}} & = & m_{f}(\phi, \Lambda, \delta, N, \gamma - c).
	\end{eqnarray*}
	It implies that $m_{f}(\psi, \Lambda, \delta, \gamma) = m_{f}(\phi, \Lambda, \delta, N, \gamma - c)$. If $\gamma > 0$ is such that 
	$m_{f}(\psi, \Lambda, \delta, \gamma)=0$, then $m_{f}(\phi, \Lambda, \delta, N, \gamma - c)=0 \Rightarrow \gamma - c \geq P_{\Lambda}(\phi,\delta)$ or
	$\gamma \geq P_{\Lambda}(\phi,\delta) + c \Rightarrow P_{\Lambda}(\psi,\delta) \geq P_{\Lambda}(\phi,\delta) + c$.
	
	In the other hand, if $m_{f}(\phi, \Lambda, \delta, N, \beta) = 0$, then $m_{f}(\psi, \Lambda, \delta, \beta + c) = 0$ and 
	$P_{\Lambda}(\psi,\delta) \leq \beta + c \Rightarrow P_{\Lambda}(\psi,\delta) \leq P(\phi,\delta) + c$. It means that
	$P_{\Lambda}(\psi,\delta) = P(\phi,\delta) + c$ or $P_{\Lambda}(\psi) = P(\phi) + c$.
\end{proof}

The previous lemma also shows that if $\phi$ is a hyperbolic (and zooming) potential, so is $\phi + c, \text{for all} \, \, c \in \mathbb{R}$. We can also obtain the following lemma.

\begin{lemma}
	If $\phi \leq \psi$, then $P_{\Lambda}(\phi) \leq P_{\Lambda}(\psi)$. 
\end{lemma}
\begin{proof}
	If $\phi \leq \psi$, we obtain $R_{n,\delta}\phi(x) \leq R_{n,\delta}\psi(x)$. So,
	\begin{eqnarray*}
	\displaystyle m_{f}(\phi, \Lambda, \delta, N, \gamma) & = & \inf_{\mathcal{U} \in \mathcal{F}_{N}(\Lambda)} \bigg{\{} \sum_{B_{\delta}(y,n) \in \mathcal{U}}
	e^{-\gamma n + R_{n,\delta}\phi(y)} \bigg{\}} \\ 
    \leq \inf_{\mathcal{U} \in \mathcal{F}_{N}(\Lambda)} \bigg{\{} \sum_{B_{\delta}(y,n) \in \mathcal{U}}
	e^{-\gamma n + R_{n,\delta}\psi(y)} \bigg{\}} & = & m_{f}(\psi, \Lambda, \delta, N, \gamma),
	\end{eqnarray*}
	which implies that
	\[
	\displaystyle m_{f}(\phi, \Lambda, \delta, \gamma) = \lim_{N \to + \infty} m_{f}(\phi, \Lambda, \delta, N, \gamma) \leq 
	\lim_{N \to + \infty} m_{f}(\psi, \Lambda, \delta, N, \gamma) = m_{f}(\psi, \Lambda, \delta, \gamma).
	\]
	It implies that $m_{f}(\phi, \Lambda, \delta, \gamma) = 0$ if $m_{f}(\psi, \Lambda, \delta, \gamma) = 0$, that is,
	\[
	P_{\Lambda}(\phi, \delta) = \inf \{\gamma  \mid m_{f}(\phi, \Lambda, \delta, \gamma) = 0\} \leq
	\inf \{\gamma  \mid m_{f}(\psi, \Lambda, \delta, \gamma) = 0\} = P_{\Lambda}(\psi, \delta).
	\]
	Finally,
	\[
	P_{\Lambda}(\phi) = \lim_{\delta \to 0} P_{\Lambda}(\phi, \delta) \leq \lim_{\delta \to 0} P_{\Lambda}(\psi, \delta) = P_{\Lambda}(\psi).
	\]
\end{proof}

With the previous lemma we can obtain the following examples.

\begin{example}
	Let $\varphi: M \to \mathbb{R}$ be a hyperbolic potential and $\phi:M \to \mathbb{R}$ such that
	\[
	\max \phi - \min \phi < P_{\Lambda}(\varphi) - P_{\Lambda^{c}}(\varphi). 
	\]
	It implies that
	\[
	P_{\Lambda^{c}}(\varphi + \phi)  \leq  P_{\Lambda^{c}}(\varphi + \max \phi) = P_{\Lambda^{c}}(\varphi) + \max \phi <
	\]
	\[
	< P_{\Lambda}(\varphi) + \min \phi = P_{\Lambda}(\varphi + \min \phi) \leq P_{\Lambda}(\varphi + \phi). 
	\]
	So, $\varphi + \phi$ is a hyperbolic potential.
	
	If $|t|\leq 1$, we also have 
	\[
	\max t\phi - \min t\phi < P_{\Lambda}(\varphi) - P_{\Lambda^{c}}(\varphi). 
	\]
	and $\varphi + t\phi$ is also a hyperbolic potential.
	
	In particular, since the null potential is hyperbolic, if we have
	\[
	\max \phi - \min \phi < P_{\Lambda}(0)=P(0)=h(f),
	\]
	then $\phi$ is also a hyperbolic potential.
\end{example}

\begin{example}
	Now, for Viana maps, we construct a potential with uniformly bounded Birkhoff sums. 
	
	Let $B$ be an open set and $V=f^{-1}(B)$ such that $V \cap B = \emptyset$ and $V \cap \mathcal{C} = \emptyset$, where $\mathcal{C}$ is the critical set. Let $\phi : \overline{B} \to \mathbb{R}$ be a $C^{\infty}$ function such that $\phi_{\mid \partial B} \equiv 0$ and we define a potential $\varphi : X \to \mathbb{R}$ as
	\[
	\varphi(x) =
	\left\{
	\begin{array}{cc}
		\phi(x), & \text{if} \,\, x  \in \displaystyle B \\
		-\phi(f(x)), &  \text{if} \,\, x \in V\\
		0, &  \text{if} \,\, x \in (V \cup B)^{c}\\    
	\end{array}
	\right.  
	\]
	\begin{claim}
		The Birkhoff sums $S_{n}\varphi$ are uniformly bounded.
	\end{claim}
	\begin{proof}
		For $x \in V$, we have that
		\begin{eqnarray*}
		S_{n}\varphi(x)  & = & \varphi(x) + \varphi(f(x)) + \dots + \varphi(f^{n-1}(x)) \\ 
		& = & -\phi(f(x)) + \phi(f(x)) + 0 + \dots + \phi(f^{n-1}(x)) \leq \sup\phi,
		\end{eqnarray*}
		For $x \in B$, we have
		\[
		S_{n}\varphi(x) = \varphi(x) + \varphi(f(x)) + \dots + \varphi(f^{n-1}(x)) \leq \sup\phi - \inf\phi,
		\]
		if $f^{n-1}(x) \in V$ and it is equal to $\phi(x)$, otherwise.
				
		For $x \in (V \cup B)^{c}$, we have at most the same estimate for $S_{n}\varphi(x)$ because the orbit of $x$ may intersect $V \cup B$.
	\end{proof}
	
	So, the Birkhoff sums are uniformly bounded and Lemma \ref{BIR} guarantees that $\varphi$ is hyperbolic. Moreover, $\varphi$ is H\"older, which means that we have existence and finiteness of equilibrium state.
\end{example}

\begin{remark}
We observe that this section is also developed for nonexponential contractions, that is, for general zooming systems with the mild condition $\alpha_{n}(r)\leq ar$ for some $a \in (0,1)$. In the case of exponential contractions, we emphasize the relation with the work in \cite{AOS}. The novelty here is the generality of contractions beyond the exponential context and the uniqueness of equilibrium state.
\end{remark}
	
\section{Pseudo-geometric Potentials}\label{Pseudo}

In this section, if $f:M \to M$ is a zooming map and if the contraction $(\alpha_{n})_{n}$ satisfies $\alpha_{n}(r) \leq ar$ for some $a \in (0,1)$, every $n \in \mathbb{N}$ and every $r \in [0,+\infty)$ (Lipschitz contractions, for example) and a zooming measure with bounded distortion, we show that the pseudo-geometric potential that we introduced in this paper is zooming for $t < t_{0}$ (defined below) and its induced potential is locally Hölder. 

\subsection{Pseudo-geometric potentials and zooming measures}

Now, we show that for $t < 0$ small enough a necessary condition for a measure to be an equilibrium state for the potential $\phi_{t} = -t \log J_{\mu} f$ is being zooming.

\begin{proposition}
	If $f:M \to M$ is a zooming map and if the contraction $(\alpha_{n})_{n}$ satisfies $\alpha_{n}(r) \leq ar$ for some $a \in (0,1)$, every $n \in \mathbb{N}$ and every $r \in [0,+\infty)$ (Lipschitz contractions, for example) and the Jacobian of the reference measure has bounded distortion, then an ergodic measure $\nu_{0}$ has free energy high enough for $\phi_{t} = -t \log J_{\mu}f $ with $t < t_{0} \leq 0$, where 
	\[
	\displaystyle t_{0} : = \max_{\mu_{0} \in \mathbb{A}} \bigg{\{} \frac{h(f)}{-\int \log J_{\mu} f d\mu_{0}} \bigg{\}} \leq 0,
	\]
	then $\nu_{0}$ is a zooming measure. In particular, the pseudo-geometric potential is zooming.
\end{proposition}

\begin{proof}
	In fact, let $\eta$ be an ergodic measure with $\eta(\Lambda) = 0$. Since $\log J_{\mu}f \leq 0$ on $\Lambda^{c}$, we have, for $t < t_{0} \leq 0$
	\begin{eqnarray*}
	h_{\eta}(f) + \int \phi_{t} d\eta  & = & h_{\eta}(f) -t \int \log J_{\mu} f d\eta = h_{\eta}(f) \\ 
	\leq h(f) < -t\int \log J_{\mu} f d\mu_{0} & = & \int -t \log J_{\mu} f d\mu_{0} \leq h_{\mu_{0}}(f) + \int \phi_{t} d \mu_{0}.
	\end{eqnarray*}
	It means that a measure must be zooming to have any chance of being an equilibrium state. In other words, under these assumptions the potential is zooming.
\end{proof}

\begin{lemma}
	If $f:M \to M$ is a zooming map and if the contraction $(\alpha_{n})_{n}$ satisfies $\alpha_{n}(r) \leq ar$ for some $a \in (0,1)$, every $n \in \mathbb{N}$ and every $r \in [0,+\infty)$ (Lipschitz contractions, for example), then the induced pseudo-geometric potential $\overline{\phi}_{t}$ is locally H\"older.
\end{lemma}

\begin{proof}
	Given a zooming system $f: M \to M$ with contractions as above and a measure with bounded distortion, we have that, $\text{for all} \, \, y,z \in V_{n}(x)$, where $n$ is a zooming time for $x$ and $V_{n}(x)$ is a zooming pre-ball, the following properties, $\text{for all} \, \, y,z \in V_{n}(x)$, by the law of expansion and bounded distortion:
	\[
	d(f^{i}(y),f^{i}(z)) \leq \sigma^{n-i}(d(f^{n}(y),f^{n}(z))), \,\, \text{for all} \, \, \,\, 0 \leq i \leq n.
	\]
	\[
	\exists \rho>0 \,\, \text{such that} \,\,\mid \log  J_{\mu}f^{n}(y)  - \log  J_{\mu}f^{n}(z) \mid \leq \rho d(f^{n}(y),f^{n}(z))
	\]
	Recall that the elements of the partition were constructed in such a way that the inducing times are also zooming times, constants in each partition element. 
	
	Given $y,z \in P_{i_{0}}$ arbitrarily, we have $R_{i_{0}}(y)=R_{i_{0}}(z):=n_{0}$. From the properties above, we get
	\[
	\mid \log J_{\mu}f^{n_{0}}(y) - \log  J_{\mu} f^{n_{0}}(z) \mid \leq \rho d(f^{n_{0}}(y),f^{n_{0}}(z))
	\]
	It means that 
	\[
	\mid \log J_{\mu}F(y) - \log  J_{\mu}F(z) \mid \leq \rho d(F(y),F(z))
	\]
	Also,
	\[
	d(y,z) \leq \alpha_{n_{0}}(d(F(y),F(z))).
	\]
	Analogously,
	\[
	F(y),F(z) \in P_{i_{1}} \Rightarrow R_{i_{1}}(F(y)) = R_{i_{1}}(F(z)) = n_{1} \Rightarrow 
	\]
	\[
	\Rightarrow d(F(y),F(z)) \leq \alpha_{n_{1}}(d(F^{2}(y),F^{2}(z))).
	\]
	In general,
	\[
	F^{j}(y),F^{j}(z) \in P_{i_{j}} \Rightarrow R_{i_{j}}(F^{j}(y)) = R_{i_{j}}(F^{j}(z)) = n_{j} \Rightarrow 
	\]
	\[
	\Rightarrow d(F^{j}(y),F^{j}(z)) \leq \alpha_{n_{j}}(d(F^{j+1}(y),F^{j+1}(z))).
	\]
	Finally, we get
	\[
	\mid \log J_{\mu}F(y) - \log  J_{\mu}F(z) \mid \leq \rho d(F(y),F(z))
	\leq \dots \leq \rho \alpha_{n_{k}}\circ \dots \circ \alpha_{n_{1}}\circ\alpha_{n_{0}} (\delta) \leq  \rho a^{k+1} \delta.
	\]
	So,
	\[
	\mid -t \log J_{\mu}F(y) - (-t) \log  J_{\mu}F(z) \mid \leq \mid  t \mid\rho a^{k} \delta, \text{for all} \, \, \,\, y,z \in P_{i_{0}}
	\]
	It implies that
	\[
	V_{n}(\overline{\phi}_{t}) = \sup \{\mid\overline{\phi}_{t}(y) - \overline{\phi}_{t}(z)\mid : y_{0} = z_{0}, y_{1} = z_{1}, \dots, y_{k} = z_{k} \} \leq \mid t \mid \rho \delta a^{k}
	\]
	Then, the induced potential is locally H\"older.
\end{proof}

Given $\phi_{t} : M \to \mathbb{R}$ a pseudo-geometric potential, we set $\varphi_{t} : = \phi_{t} - P_{f}(\phi_{t})$. We observe that the induced potential $\overline{\varphi}_{t} = -t\log J_{\mu} F - R P_{f}(\phi_{t})$ is also locally H\"older.

Since the potential $\phi_{t}$ is zooming for $t < t_{0}$, we also have that $\varphi_{t}$ is zooming. Since $\overline{\varphi}_{t}$ is locally Hölder, we can proceed analogously to the case where $\phi$ is Hölder to obtain finitely many ergodic equilibrium states.

\section{Existence of pseudo-conformal measures}\label{Conformal}

Now, we proceed to show that the conformal measure which we showed to exist for $\overline{\varphi}_{t}, t < t_{0}$ can be projected to the original map to obtain a pseudo-conformal measure. We will extend the technique of G. Iommi and M. Todd in \cite{IT1}. For this section, we consider  $f:M \to M$  a zooming map and if the contraction $(\alpha_{n})_{n}$ satisfying $\alpha_{n}(r) \leq ar$ for some $a \in (0,1)$, every $n \in \mathbb{N}$ and every $r \in [0,+\infty)$ (Lipschitz contractions, for example) and a zooming measure with bounded distortion and denote $(U,F,R)$ denotes an inducing scheme constructed on an open set $U$, with induced maps $F$ and return time $R$.

Firstly, we will show that we can find an inducing scheme that satisfies the Condition (*) below.

We denote by $(U,F)^{\infty}$ the set of points $x \in M$ for which there exists $k \in \mathbb{N}$ such that $R(F^{n}(f^{k}(x)) < \infty, \text{for all} \, \, n \in \mathbb{N}$. 

\textbf{Condition (*):} If $x \in (U,F)^{\infty}$ is such that $f^{k}(y)=f^{k'}(y')=x$ for $y,y' \in (U,F)^{\infty}$ and $k,k' \in \mathbb{N}$, then, there exists $n \in \mathbb{N}$ such that $k + n, k' + n$ are inducing times for $y, y'$, respectively.

The inducing schemes constructed by Pinheiro in \cite{Pi1} may not satisfy the Condition (*). But we will make a change in order to obtain an inducing scheme satisfying this condition.

\begin{lemma}
There exists an inducing scheme that satisfies the Condition (*).
\end{lemma}

\begin{proof}
In fact, we remind that the collection of regular pre-images of a hyperbolic ball is a dynamically closed family and our partition was built in such a way that $P_{i} = f^{-R _{i}}(U)$, a regular pre-image of a subset of a hyperbolic ball. 
Moreover, our collection is nested, which means that there are not linked pre-images of a same hyperbolic ball. (See Pinheiro \cite{Pi1})

Given such an inducing scheme, we only allow pairs of elements of the partition $P_{i},P_{j}$ such that $f^{m}(P_{j}) \subset P_{i}$ for some $m \in \mathbb{Z}$ implies $f^{m}(P_{j}) = P_{i}$. Otherwise, we erase $P_{j}$ from our inducing scheme as well $F_{n}^{-1}(P_{j})$ for all $n \in \mathbb{N}$, where $F_{n}: P_{n} \to U$ is the induced map restricted to the element $P_{n}$.

\begin{claim}
The remaining elements of the partition all together form an inducing scheme satisfying Condition (*).
\end{claim}

Given $x,y,y' \in (U,F)^{\infty}$ and $k,k' \in \mathbb{N}$, such that $k<R_{i}, k'<R_{j}$ and $f^{k}(y)=x=f^{k'}(y')$ with $y \in P_{i}, y' \in P_{j}$, we claim that $f^{k}(P_{i}) =f^{k'}(P_{j})$. 

Since $x \in f^{k}(P_{i}) \cap f^{k'}(P_{j})$, we have $f^{k}(P_{i}) \cap f^{k'}(P_{j}) \neq \emptyset$. It implies that either $f^{k}(P_{i}) \subset f^{k'}(P_{j})$ or $f^{k'}(P_{j}) \subset f^{k}(P_{i})$, because $U$ is nested and can not have linked pre-images. In other words, either $f^{k - k'}(P_{i}) \subset P_{j}$ or $f^{k'- k}(P_{j}) \subset P_{i}$.

Because of our assumption, it means that either $f^{k - k'}(P_{i}) = P_{j}$ or $f^{k'- k}(P_{j}) = P_{i}$ and either $f^{R_{j} + k - k'}(P_{i}) \subset f^{R_{j}}(P_{j}) = U$ or $f^{R_{i} + k'- k}(P_{j}) \subset f^{R_{i}}(P_{i}) = U$. Then, we obtain $R_{j} + k - k' = R_{i}$ or $R_{i} + k'- k = R_{j}$. So, in any case, we put $n = R_{j} - k' = R_{i} - k$ and we have $R_{i} = k + n$ and $R_{j} = k' + n$.

So, there always exists $n \in \mathbb{N}$ such that $k + n = R_{i}, k' + n = R_{j}$. Then, both $k + n$ and $k' + n$ are inducing times and our inducing scheme satisfies the Condition (*).
\end{proof}

Now we follow ideas of a theorem proved in \cite{IT1} to project conformal measures. We consider $\phi = \varphi_{t}, t < t_{0}$.

\begin{theorem}
Suppose that $f:M \to M$ is a zooming map with general contractions and a zooming measure with bounded distortion and $(U,F,R)$ an inducing scheme satisfying condition $(*)$. Let us consider $\phi : M \to [-\infty,\infty]$ and its induced version $\Phi : U \to [-\infty,\infty)$ with a $\Phi$-conformal probability measure $m_{\Phi}$. Then $m_{\Phi}$ projects to a $\sigma$-finite $\phi$-conformal measure $\nu_{\phi}$ and it is finite if $\phi$ is bounded from below.
\end{theorem}

\begin{proof}
First, note that $\Phi < \infty$ implies that for any set $A \subset U, m_{\Phi}(A)>0$ implies $m_{\Phi}(F(A)) > 0$. This means that no set of positive measures can leave $\cup_{i} P_{i}$ under iteration of $F$, i.e., for all $k \in \mathbb{N}, m_{\Phi}(F^{-k}(\cup_{i}P_{i})) = 1$ and thus $m_{\Phi}(U\cap(X,F)^{\infty}) = 1$. We will spread the measure $m_{\Phi}$ onto $(U,F)^{\infty}$ as follow.

Suppose that $x \in (U,F)^{\infty}$ is contained in some $f^{k}(P_{i})$ for $0 \leq k \leq R_{i} - 1$. There may be many such pairs $(i,k)$, but we pick one arbitrarily and then later show that we could have chosen any and obtained the same result. There exists a unique $y \in P_{i}$ such that $f^{k}(y) = x$. We define $\nu_{\phi}$ so that for any $j \in \mathbb{N}$,
\[
d\nu_{\phi}(f^{j}(x)) = e^{-S_{k + j}\phi(y)} dm_{\Phi}(y).
\]
If the measure is well defined, this gives a conformal measure locally, since for $j \in \mathbb{N}$,
\[
d\nu_{\phi}(f^{j}(x)) = e^{-S_{k + j}\phi(y)} dm_{\Phi}(y) = e^{-S_{j}\phi(x)} e^{-S_{k}\phi(y)} dm_{\Phi}(y) = e^{-S_{j}\phi(x)}d\nu_{\phi}(x).
\]
Note that if $k + j = R_{i}$, then $f^{k + j}(y) = F(y)$ and
\begin{eqnarray*}
d\nu_{\phi}(F(y))  & = & d\nu_{\phi}(f^{R(y)}(y)) = d\nu_{\phi}(f^{k + j}(y)) = d\nu_{\phi}(f^{j}(x)) \\ 
& = & e^{-S_{k + j}\phi(y)}dm_{\Phi}(y) = e^{-\Phi(y)}dm_{\Phi}(y) = dm_{\Phi}(F(y)),
\end{eqnarray*}
by the $\Phi$-conformality of $m_{\Phi}$. 

This also extends to the case when $k + j = R^{p}(y)$ for $p \in \mathbb{N}$, where we obtain
\begin{eqnarray*}
d\nu_{\phi}(F^{p}(y)) & = & d\nu_{\phi}(f^{R_{p}(y)}(y)) = d\nu_{\phi}(f^{k+j}(y)) = d\nu_{\phi}(f^{j}(x)) \\
& = & e^{-S_{k + j}\phi(y)}dm_{\Phi}(y) = e^{-S_{p}\Phi(y)}d_{\Phi}(y) = dm_{\Phi}(F^{p}(y)).
\end{eqnarray*}
In order to prove that the procedure given above is well defined we need to check that the same measure is assigned at $x$ when there is $i' \neq i$ and $1 \leq k' \leq R_{i'} -1$ such that $x$ is also contained in $f^{i'}(P_{i'})$. If we let $\nu'_{\phi}$ be the measure at $x$ obtained analogously to $\nu_{\phi}$ but with $i'$ and $k'$ in place of $i$ and $k$, respectively, and some point $y' \in P_{i'}$ in place of $y \in P_{i}$, we must show $\nu'_{\phi} = \nu_{\phi}$.

By condition $(*)$ the exists $n$ such that $k + n$ and $k' + n$ are inducing times for $y$ and $y'$, respectively. 

We have that
\[
d\nu_{\phi}(f^{n}(x)) = d\nu_{\phi}(f^{k+n}(y)) = d\nu_{\phi}(F(y)) = dm_{\Phi}(F(y)) = dm_{\Phi}(f^{k+n}(y)) = dm_{\Phi}(f^{n}(x))
\]
and analogously
\[
d\nu'(f^{n}(x)) = dm_{\Phi}(f^{n}(x)),
\]
Moreover, since $m_{\Phi}$ is conformal
\begin{eqnarray*}
d\nu_{\phi}(x) & = & e^{-S_{k}\phi(y)}dm_{\Phi}(y) = e^{S_{n}\phi(x)} e^{-S_{n}\phi(x)} e^{-S_{k}\phi(y)} dm_{\Phi}(y) \\
& = & e^{S_{n}\phi(x)} e^{-S_{k+n}\phi(y)} dm_{\Phi}(y) = e^{S_{n}\phi(x)} e^{-\Phi(y)} dm_{\Phi}(y) \\
& = & e^{S_{n}\phi(x)} dm_{\Phi}(F(y)) = e^{S_{n}\phi(x)} d\nu_{\phi}(F(y)) = e^{S_{n}\phi(x)} d\nu_{\phi}(f^{n}(x)).
\end{eqnarray*}
Analogously,
\[
d\nu'_{\phi}(x) = e^{S_{n}\phi(x)} d\nu'_{\phi}(f^{n}(x)).
\]
Therefore,
\[
d\nu'_{\phi}(x) = e^{S_{n}\phi(x)}d\nu'_{\phi}(f^{n}(x)) = e^{S_{n}\phi(x)}dm_{\Phi}(f^{n}(x)) = e^{S_{n}\phi(x)}d\nu_{\phi}(f^{n}(x)) = d\nu_{\phi}(x),
\]
so $\nu'_{\phi} = \nu_{\phi}$, as required.

For $\sigma$-finiteness of $\nu_{\phi}$, notice that for any $x \in \cup_{i \geq 1} \cup_{k=0}^{R_{i} - 1} f^{k}(P_{i})$, we can choose a single element $f^{k}(P_{i})$ as above containing $x$ to apportion measure at $x$. Since the resulting measure is finite, $\nu_{\phi}$ is $\sigma$-finite.

To prove that $\nu_{\phi}$ is actually finite, we remind that $f$ is an topologically exact map: for all open set $V \subset M$ there exists $m \in \mathbb{N}$ such that $\displaystyle f^{m}(V) = M$. We can take $V = U$.

We have that
\begin{eqnarray*}
U & = & \Bigg{(}\bigcup_{R_{i} \leq m} P_{i} \Bigg{)} \cup \Bigg{(}\bigcup_{R_{i} > m} P_{i}\Bigg{)} \cup \Bigg{(}\bigcup_{i \geq 1} \partial P_{i} \Bigg{)}\\
& \Rightarrow &
f^{m}(U) \subset U \cup \Bigg{(}\bigcup_{R_{i} \leq m} \bigcup_{k=0}^{R_{i} - 1} f^{k}(P_{i}) \Bigg{)} \cup \Bigg{(}\bigcup_{R_{i} > m} f^{m}(P_{i})\Bigg{)} \cup f^{m}\Bigg{(}\bigcup_{i \geq 1} \partial P_{i} \Bigg{)} \\
&\Rightarrow & \nu_{\phi}(f^{m}(U)) = \nu_{\phi} \Bigg{(} U \cup\Bigg{(}\bigcup_{R_{i} \leq m} \bigcup_{k=0}^{R_{i} - 1} f^{k}(P_{i}) \Bigg{)} \cup \Bigg{(}\bigcup_{R_{i} > m} f^{m}(P_{i})\Bigg{)}\cup f^{m}\Bigg{(}\bigcup_{i \geq 1} \partial P_{i} \Bigg{)}\Bigg{)} \\
& \leq & \nu_{\phi}(U)  + \sum_{R_{i} \leq m} \sum_{k=0}^{R{i} - 1} \nu_{\phi}(f^{k}(P_{i})) 
 + \sum_{R{i} > m} \nu_{\phi}(f^{m}(P_{i})) + \sum_{i \geq 1} \nu_{\phi}(f^{m}(\partial P_{i}))\\ 
& = & \nu_{\phi}(U)  +
\sum_{R_{i} \leq m} \sum_{k=0}^{R{i} - 1} \int_{P_{i}} e^{-S_{k}\phi}d\nu_{\phi}
 + \sum_{R{i} > m} \int_{P_{i}} e^{-S_{m}\phi}d\nu_{\phi}\\
& \leq & \nu_{\phi}(U)  +
\sum_{R_{i} \leq m} \sum_{k=0}^{R{i} - 1} e^{k \sup(-\phi)} \nu_{\phi}(P_{i})
 + \sum_{R{i} > m} e^{m \sup(-\phi)} \nu_{\phi}(P_{i}) \\
& \leq &\nu_{\phi}(U)  + \sum_{R_{i} \leq m} \sum_{k=0}^{R{i} - 1} e^{k \sup(-\phi)} \nu_{\phi}(P_{i})
 + e^{m \sup_{U}(-\phi)} \nu_{\phi}(U).
\end{eqnarray*}
We used the conformality of $\nu_{\phi}$ and that $\sup(-\phi) < \infty$. Also, $\nu_{\phi}(f^{m}(\partial P_{i})) = 0$ since the measure is spread on $\cup_{i = 1}^{\infty} \cup_{k =0}^{R_{i} - 1}f^{k}(P_{i})$. The sum is finite because there are only finitely many finite terms. Now, $\nu_{\phi}(\cup_{i = 1}^{\infty} \cup_{k =0}^{R_{i} - 1}f^{k}(P_{i})) \leq \nu_{\phi}(M) = \nu_{\phi}(f^{m}(U)) < \infty$ and the projected measure $\nu_{\phi}$ is actually finite.
\end{proof}

\section{Bounded Distortion Potentials}\label{Distortion}

In this section, if $f:M \to M$ is a zooming map and if the contraction $(\alpha_{n})_{n}$ satisfies $\alpha_{n}(r) \leq ar$ for some $a \in (0,1)$, every $n \in \mathbb{N}$ and every $r \in [0,+\infty)$ (Lipschitz contractions, for example) we have that the bounded distortion potentials $\psi_{t}$, that we introduce here, are zooming for $t > t_{0}$ (defined below) and its induced potential is locally Hölder. 

The following definition is based on the notion of measure with bounded distortion and its Jacobian.

\begin{definition}
	Let $\mu$ be the zooming reference measure and $\mathbb{A}$ the set of ergodic absolutely continuous invariant measures. We say that a measurable function $\psi : M \to \mathbb{R}$ such that $\int \psi d\mu > 0$ has \textbf{\textit{bounded distortion}} if  there exists $\rho > 0$ such that
	\[
	\Bigg{|}\sum_{i=0}^{n-1}\psi(f^{i}(y)) - \sum_{i=0}^{n-1}\psi(f^{i}(z))\Bigg{|} \leq \rho d(f^{n}(y), f^{n}(z)),
	\]
	for every $y,z \in V_{n}(x)$, $\mu$-almost everywhere  $x \in M$, for every zooming time $n$ of $x$. For $t \in \mathbb{R}$, we call $\psi_{t} = t\psi$ a \textbf{\textit{bounded distortion potential}}.
\end{definition}

\subsection{Bounded distortion potentials and zooming measures}

Now, we show that for $t > 0$ small enough a necessary condition for a measure to be an equilibrium state for the potential $\psi_{t} = t \psi$ is being zooming.

\begin{proposition}
	If $f:M \to M$ is a zooming map and if the contraction $(\alpha_{n})_{n}$ satisfies $\alpha_{n}(r) \leq ar$ for some $a \in (0,1)$, every $n \in \mathbb{N}$ and every $r \in [0,+\infty)$ (Lipschitz contractions, for example), then an ergodic measure $\nu_{0}$ has free energy high enough for $\psi_{t} = t \psi $ with $t > t_{0} \geq 0$, where 
	\[
	\displaystyle t_{0} : = \max_{\mu_{0} \in \mathbb{A}} \bigg{\{} \frac{h(f)}{\int \psi d\mu_{0}} \bigg{\}} \geq 0,
	\]
	then $\nu_{0}$ is a zooming measure. In particular, the bounded distortion  potential is zooming.
\end{proposition}

\begin{proof}
	In fact, let $\eta$ be an ergodic measure with $\eta(\Lambda) = 0$. Since $\psi \geq 0$, we have, for $t > t_{0} \geq 0$
	\begin{eqnarray*}
		h_{\eta}(f) + \int \psi_{t} d\eta  & = & h_{\eta}(f)  + t \int \psi d\eta = h_{\eta}(f) \\ 
		\leq h(f) < t\int \psi d\mu_{0} & = & \int t \psi d\mu_{0} \leq h_{\mu_{0}}(f) + \int \phi_{t} d \mu_{0}.
	\end{eqnarray*}
	It means that a measure must be zooming to have any chance of being an equilibrium state. In other words, under these assumptions the potential is zooming.
\end{proof}

\begin{lemma}
	If $f:M \to M$ is a zooming map and if the contraction $(\alpha_{n})_{n}$ satisfies $\alpha_{n}(r) \leq ar$ for some $a \in (0,1)$, every $n \in \mathbb{N}$ and every $r \in [0,+\infty)$ (Lipschitz contractions, for example). Then the induced bounded distortion potential $\overline{\phi}_{t}$ is locally H\"older.
\end{lemma}

\begin{proof}
	Given a zooming system $f: M \to M$ with contractions as above and a function $\psi$ with bounded distortion, we have that, $\text{for all} \, \, y,z \in V_{n}(x)$, where $n$ is a zooming time for $x$ and $V_{n}(x)$ is a zooming pre-ball, the following properties, $\text{for all} \, \, y,z \in V_{n}(x)$, by the law of expansion and bounded distortion:
	\[
	d(f^{i}(y),f^{i}(z)) \leq \sigma^{n-i}(d(f^{n}(y),f^{n}(z))), \,\, \text{for all} \, \, \,\, 0 \leq i \leq n.
	\]
	\[
	\exists \,\, \rho> 0 \,\, \text{such that} \,\, \Bigg{|}\sum_{i=0}^{n-1}\psi(f^{i}(y)) - \sum_{i=0}^{n-1}\psi(f^{i}(z))\Bigg{|} \leq \rho d(f^{n}(y), f^{n}(z)).
	\]
	Recall that the elements of the partition were constructed in such a way that the inducing times are also zooming times, constants in each partition element. 
	
	Given $y,z \in P_{i_{0}}$ arbitrarily, we have $R_{i_{0}}(y)=R_{i_{0}}(z):=n_{0}$. From the properties above, we get
	\[
	\Bigg{|}\sum_{i=0}^{n_{0}-1}\psi(f^{i}(y)) - \sum_{i=0}^{n_{0}-1}\psi(f^{i}(z))\Bigg{|} \leq \rho d(f^{n_{0}}(y), f^{n_{0}}(z)),
	\]
	It means that 
	\[
	\mid \overline{\psi}(y) - \overline{\psi}(z) \mid \leq \rho d(F(y),F(z)).
	\]
	Also,
	\[
	d(y,z) \leq \alpha_{n_{0}}(d(F(y),F(z))).
	\]
	Analogously,
	\[
	F(y),F(z) \in P_{i_{1}} \Rightarrow R_{i_{1}}(F(y)) = R_{i_{1}}(F(z)) = n_{1} \Rightarrow 
	\]
	\[
	\Rightarrow d(F(y),F(z)) \leq \alpha_{n_{1}}(d(F^{2}(y),F^{2}(z))).
	\]
	In general,
	\[
	F^{j}(y),F^{j}(z) \in P_{i_{j}} \Rightarrow R_{i_{j}}(F^{j}(y)) = R_{i_{j}}(F^{j}(z)) = n_{j} \Rightarrow 
	\]
	\[
	\Rightarrow d(F^{j}(y),F^{j}(z)) \leq \alpha_{n_{j}}(d(F^{j+1}(y),F^{j+1}(z))).
	\]
	Finally, we get
	\[
	\mid \overline{\psi}(y) - \overline{\psi}(z) \mid \leq \rho d(F(y),F(z))
	\leq \dots \leq \rho \alpha_{n_{k}}\circ \dots \circ \alpha_{n_{1}}\circ\alpha_{n_{0}} (\delta) \leq  \rho a^{k+1} \delta.
	\]
	So,
	\[
	\Bigg{|}\sum_{i=0}^{n_{0}-1}\psi(f^{i}(y)) - \sum_{i=0}^{n_{0}-1}\psi(f^{i}(z))\Bigg{|} \leq \rho a^{k} \delta, \text{for all} \, \, \,\, y,z \in P_{i_{0}}.
	\]
	It implies that
	\[
	V_{n}(\overline{\psi}_{t}) = \sup \{\mid\overline{\psi}_{t}(y) - \overline{\psi}_{t}(z)\mid : y_{0} = z_{0}, y_{1} = z_{1}, \dots, y_{k} = z_{k} \} \leq  |t|  \rho \delta a^{k}
	\]
	Then, the induced potential is locally H\"older.
\end{proof}

Given $\psi_{t} : M \to \mathbb{R}$ a bounded distortion potential, we set $\varphi_{t} : = \psi_{t} - P_{f}(\phi_{t})$. We observe that the induced potential $\overline{\varphi}_{t} = -t\log J_{\mu} F - R P_{f}(\phi_{t})$ is also locally H\"older.

Since the potential $\psi_{t}$ is zooming for $t > t_{0}$, we also have that $\varphi_{t}$ is zooming. Since $\overline{\varphi}_{t}$ is locally Hölder, we can proceed analogously to the case where $\phi$ is Hölder to obtain finitely many ergodic equilibrium states.

\begin{remark}
	We observe that the pseudo-geometric potentials $\phi_{t} = -t\log J_{\mu}f$ are examples of bounded distortion potentials.
\end{remark}

\begin{example}
    Let $u : M \to \mathbb{R}$ be a Lipschitz function and $K > 0$ such that it holds that $\psi = u - u \circ f + K \geq 0$. We claim that $\psi$ has bounded distortion. In fact, we have for every $w \in M$
    \[
    \sum_{i=0}^{n-1}\psi(f^{i}(w)) =  u(w) - u(f^{n}(w)) + nK 
    \]
    and for every $y,z \in V_{n}(x)$
    \[
    \Bigg{|}\sum_{i=0}^{n-1}\psi(f^{i}(y)) - \sum_{i=0}^{n-1}\psi(f^{i}(z))\Bigg{|} = \mid (u(y) - u(f^{n}(y)) + nK) - (u(z) - u(f^{n}(z)) + nK)\mid \leq 
    \]
    \[
    \mid u(y) - u(z) \mid + \mid u(f^{n}(y)) - u(f^{n}(z)) \mid \leq d(y,z) + d(f^{n}(y), f^{n}(z)) \leq 2 d(f^{n}(y), f^{n}(z)). 
    \] 
    Moreover, $\int \psi d\mu = K > 0$. Then, $\psi$ has bounded distortion with $\rho = 2$.
\end{example}	

\section{Examples}\label{Examples}

In this section, we give examples of  zooming systems. We begin by defining a non-flat map. We begin by recalling the examples given in \cite{AOS}, where the expanding set is dense in $M$, the hole is empty and the map is closed.

\subsection{Viana maps} We recall the definition of the open class of maps with critical sets in dimension 2, introduced by M. Viana in \cite{V}. We skip the technical
points. It can be generalized for any dimension (See \cite{A}).

Let $a_{0} \in (1,2)$ be such that the critical point $x=0$ is pre-periodic for the quadratic map $Q(x)=a_{0} - x^{2}$. Let $S^{1}=\mathbb{R}/\mathbb{Z}$ and 
$b:S^{1} \to \mathbb{R}$ a Morse function, for instance $b(\theta) = \sin(2\pi\theta)$. For fixed small $\alpha > 0$, consider the map
\[
\begin{array}{c}
f_{0}: S^{1} \times \mathbb{R} \longrightarrow S^{1} \times \mathbb{R}\\
\,\,\,\,\,\,\,\,\,\,\,\,\,\,\,\,\,\,\,\ (\theta,x) \longmapsto (g(\theta),q(\theta,x))
\end{array}
\] 
where $g$ is the uniformly expanding map of the circle defined by $g(\theta)=d\theta
(mod\mathbb{Z})$ for some $d \geq 16$, and $q(\theta,x) = a(\theta) - x^{2}$ with $a(\theta) = a_{0} + \alpha b(\theta)$. It is easy to check that for $\alpha > 0$ 
small enough there is an interval $I \subset (-2,2)$ for which $f_{0}(S^{1} \times I)$ is contained in the interior of $S^{1} \times I$. Thus, any map $f$ sufficiently
close to $f_{0}$ in the $C^{0}$ topology has $S^{1} \times I$ as a forward invariant region. We consider from here on these maps $f$ close to $f_{0}$ restricted to 
$S^{1} \times I$. Taking into account the expression of $f_{0}$ it is not difficult to check that for $f_{0}$ (and any map $f$ close to $f_{0}$ in the $C^{2}$ topology)
the critical set is non-degenerate.

The main properties of $f$ in a $C^{3}$ neighbourhood of $f$ that we will use here are summarized below (See \cite{A},\cite{AV},\cite{Pi1}):

\begin{enumerate}
 \item[(1)] $f$ is \textbf{\textit{non-uniformly expanding}}, that is, there exist $\lambda > 0$ and a Lebesgue full measure set $H \subset S^{1} \times I$ such that 
  for every point $p=(\theta, x) \in H$, the following holds
\[
\displaystyle \limsup_{n \to \infty} \frac{1}{n} \sum_{i=0}^{n-1} \log \parallel Df(f^{i}(p))^{-1}\parallel^{-1} < -\lambda.
\]  
 \item[(2)] Its orbits have \textbf{\textit{slow approximation to the critical set}}, that is, for every $\epsilon > 0$ the exists $\delta > 0$ such that for every point
  $p=(\theta, x) \in H \subset S^{1} \times I$, the following holds 
\[
\displaystyle \limsup_{n \to \infty} \frac{1}{n} \sum_{i=0}^{n-1} - \log \text{dist}_{\delta}(p,\mathcal{C}) < \epsilon.
\]  
where 
\[
\text{dist}_{\delta}(p,\mathcal{C}) =  
\left\{ 
\begin{array}{ccc}
dist(p,\mathcal{C}), & if & dist(p,\mathcal{C}) < \delta\\
1 & if & dist(p,\mathcal{C}) \geq \delta 
\end{array}
\right.
\]  
 \item[(3)] $f$ is topologically mixing;
  
 \item[(4)] $f$ is strongly topologically transitive;
  
 \item[(5)] it has a unique ergodic absolutely continuous invariant (thus SRB) measure;
  
 \item[(6)]the density of the SRB measure varies continuously in the $L^{1}$ norm with $f$.
\end{enumerate}

\begin{remark}
We observe that this definition of non-uniformly expansion is included in ours by neighbourhoods.
\end{remark}

\subsection{Benedicks-Carleson Maps} We study a class of non-hyperbolic maps of the interval with the condition of exponential growth of the derivative at critical values, called 
\textbf{\textit{Collet-Eckmann Condition}}. We also ask the map to be $C^{2}$ and topologically mixing and the critical points to have critical order 
$2 \leq \alpha < \infty$.

Given a critical point $c \in I$, the \textbf{\textit{critical order}} of $c$ is a number $\alpha_{c} > 0$ such that 
$f(x) = f(c) \pm |g_{c}(x)| ^{\alpha_{c}}, \,\, \text{for all} \, \, x \in \mathcal{U}_{c}$ where $g_{c}$ is a diffeomorphism 
$g_{c}: \mathcal{U}_{c} \to g(\mathcal{U}_{c})$ and $\mathcal{U}_{c}$ is a neighbourhood of $c$. 

Let $\delta>0$ and denote $\mathcal{C}$ the set of critical points and $\displaystyle B_{\delta} = \cup_{c \in \mathcal{C}} (c - \delta, c + \delta)$. 
Given $x \in I$, we suppose that

\begin{itemize}
 \item \textbf{(Expansion outside $B_{\delta}$)}.  There exists $\kappa > 1 $ and $\beta > 0$ such that, if $x_{k} = f^{k}(x) \not \in B_{\delta}, \,\, 0 \leq k \leq n-1$ then $|Df^{n}(x)| \geq \kappa \delta^{(\alpha_{\max} -1)}e^{\beta n}$, where $\alpha_{\max} = \max \{\alpha_{c}, c \in \mathcal{C}\}$. Moreover, if $x_{0} \in f(B_{\delta})$ or $x_{n} \in B_{\delta}$ then $|Df^{n}(x)| \geq \kappa e^{\beta n}$.
 
 \item \textbf{(Collet-Eckmann Condition)}. There exists $\lambda > 0$ such that 
 \[
 |Df^{n}(f(c))| \geq e^{\lambda n}.
 \]
  
 \item \textbf{(Slow Recurrence to $\mathcal{C}$)}. There exists $\sigma \in (0, \lambda/5)$ such that 
 \[
 dist(f^{k}(x), \mathcal{C}) \geq e^{-\sigma k}.
 \]
\end{itemize}

\subsection{Rovella Maps}

There is a class of non-uniformly expanding maps known as \textbf{\textit{Rovella Maps}}. They are derived from the so-called \textit{Rovella Attractor},
a variation of the \textit{Lorenz Attractor}. We proceed with a brief presentation. See \cite{AS} for details.

\subsubsection{Contracting Lorenz Attractor}

The geometric Lorenz attractor is the first example of a robust attractor for a flow containing a hyperbolic singularity. The attractor is a transitive maximal invariant
set for a flow in three-dimensional space induced by a vector field having a singularity at the origin for which the derivative of the vector field at the singularity has
real eigenvalues $\lambda_{2} < \lambda_{3} < 0 < \lambda_{1}$ with $\lambda_{1} + \lambda_{3} > 0$. The singularity is accumulated by regular orbits which prevent the 
attractor from being hyperbolic.

The geometric construction of the contracting Lorenz attractor (Rovella attractor) is the same as the geometric Lorenz attractor. The only difference is the condition
(A1)(i) below that gives in particular $\lambda_{1} + \lambda_{3} < 0$. The initial smooth vector field $X_{0}$ in $\mathbb{R}^{3}$ has the following properties:

\begin{itemize}
 
 \item[(A1)] $X_{0}$ has a singularity at $0$ for which the eigenvalues $\lambda_{1},\lambda_{2},\lambda_{3} \in \mathbb{R}$ of $DX_{0}(0)$ satisfy:
      \begin{itemize}
        
        \item[(i)] $0 < \lambda_{1} < -\lambda_{3}  < -\lambda_{2}$,
        
        \item[(ii)] $r > s+3$, where $r=-\lambda_{2}/\lambda_{1}, s=-\lambda_{3}/\lambda_{1}$;
      \end{itemize}

 \item[(A2)] there is an open set $U \subset \mathbb{R}^{3}$, which is forward invariant under the flow, containing the cube
 $\{(x,y,z) : \mid x \mid \leq 1, \mid y \mid \leq 1, \mid x \mid \leq 1\}$ and supporting the \textit{Rovella attractor}
 \[
 \displaystyle \Lambda_{0} = \bigcap_{t \geq 0} X_{0}^{t}(U).
 \]
 
 The top of the cube is a Poincar\'e section foliated by stable lines $\{x = \text{const}\} \cap \Sigma$ which are invariant under Poincar\'e first return map $P_{0}$.
 The invariance of this foliation uniquely defines a one-dimensional map $f_{0} : I \backslash \{0\} \to I$ for which
 \[
 f_{0} \circ \pi = \pi \circ P_{0},
 \]
 where $I$ is the interval $[-1,1]$ and $\pi$ is the canonical projection $(x,y,z) \mapsto x$;
 
 \item[(A3)] there is a small number $\rho >0$ such that the contraction along the invariant foliation of lines $x =$const in $U$ is stronger than $\rho$.
\end{itemize}

See \cite{AS} for properties of the map $f_{0}$.

\subsubsection{Rovella Parameters}

The Rovella attractor is not robust. However, the chaotic attractor persists in a measure theoretical sense: there exists a one-parameter family of positive Lebesgue measure
of $C^{3}$ close vector fields to $X_{0}$ which have a transitive non-hyperbolic attractor. In the proof of that result, Rovella showed that there is a set of parameters
$E \subset (0,a_{0})$ (that we call \textit{Rovella parameters}) with $a_{0}$ close to $0$ and $0$ a full density point of $E$, i.e.
\[
\displaystyle \lim_{a \to 0} \frac{\mid E \cap (0,a) \mid}{a} = 1,
\]
such that:

\begin{itemize}
 \item[(C1)] there is $K_{1}, K_{2} > 0$ such that for all $a \in E$ and $x \in I$
 \[
 K_{2} \mid x \mid^{s-1} \leq f_{a}'(x) \leq K_{1} \mid x \mid^{s-1},
 \]
 where $s=s(a)$. To simplify, we shall assume $s$ fixed.
 
 \item[(C2)] there is $\lambda_{c} > 1$ such that for all $a \in E$, the points $1$ and $-1$ have \textit{Lyapunov exponents} greater than $\lambda_{c}$:
 \[
 (f_{a}^{n})'(\pm 1) > \lambda_{c}^{n}, \,\, \text{for all} \, \, n \geq 0;
 \]
 
 \item[(C3)] there is $\alpha > 0$ such that for all $a \in E$ the \textit{basic assumption} holds:
 \[
 \mid f_{a}^{n-1}(\pm 1)\mid > e^{-\alpha n}, \,\, \text{for all} \, \, n \geq 1;
 \]
 
 \item[(C4)] the forward orbits of the points $\pm 1$ under $f_{a}$ are dense in $[-1,1]$ for all $a \in E$.
\end{itemize}

\begin{definition}
We say that a map $f_{a}$ with $a \in E$ is a \textbf{\textit{Rovella Map}}. 
\end{definition}

\begin{theorem}
(Alves-Soufi \cite{AS}) Every Rovella map is non-uniformly expanding. 
\end{theorem}

\subsection{Hyperbolic Times}

The idea of hyperbolic times is a key notion on the study of non-uniformly hyperbolic dynamics and it was introduced by Alves et al. 
This is powerful to get expansion in the context of non-uniform expansion. Here, we recall the basic definitions and results on hyperbolic times that we will use later on. 
We will see that this notion is an example of a Zooming Time. 

In the following, we give definitions taken from \cite{A} and \cite{Pi1}.

\begin{definition}
Let $M$ be a compact Riemannian manifold of dimension $d \geq 1$ and $f:M \to M$ a continuous map defined on $M$.
The map $f$ is called \textbf{\textit{non-flat}} if it is a local $C^{1 + \alpha}, (\alpha >0)$ diffeomorphism in the whole manifold except in a 
non-degenerate set $\mathcal{C} \subset M$. We say that $M \neq \mathcal{C} \subset M$ is a \textbf{\textit{non-degenerate set}}
if there exist $\beta, B > 0$ such that the following two conditions hold.

\begin{itemize}
  \item $\frac{1}{B} d(x,\mathcal{C})^{\beta} \leq \frac{\parallel Df(x) v\parallel}{\parallel v \parallel} \leq B d(x,\mathcal{C})^{-\beta}$ for all $v \in T_{x}M$, for every $x \in M\backslash\mathcal{C}$.
  
  For every $x, y \in M \backslash \mathcal{C}$ with $d(x,y) < d(x,\mathcal{C})/2$ we have
  
  \item $\mid \log \parallel Df(x)^{-1} \parallel - \log \parallel Df(y)^{-1} \parallel \mid \leq \frac{B}{d(x,\mathcal{C})^{\beta}} d(x,y)$.
\end{itemize}

\end{definition}

In the following, we give the definition of a hyperbolic time \cite{ALP2}, \cite{Pi1}.

\begin{definition}
(Hyperbolic times). Let us fix $0 < b = \frac{1}{3} \min\{1,1 \slash \beta\} < \frac{1}{2} \min\{1,1\slash \beta\}$. 
Given $0 < \sigma < 1$ and $\epsilon > 0$, we will say that $n$ is a $(\sigma, \epsilon)$\textbf{\textit{-hyperbolic time}} for a point $x \in M$ 
(with respect to the non-flat map $f$ with a $\beta$-non-degenerate critical/singular set $\mathcal{C})$ if for all $1 \leq k \leq n$ we have 
\[
\prod_{j=n-k}^{n-1} \|(Df \circ f^{j}(x))^{-1}\| \leq \sigma^{k} \,\, \text{and} \,\, dist_{\epsilon}(f^{n-k}(x), \mathcal{C}) \geq \sigma^{bk}.
\]
where
\[
\text{dist}_{\epsilon}(p,\mathcal{C}) =  
\left\{ 
\begin{array}{ccc}
	dist(p,\mathcal{C}), & if & dist(p,\mathcal{C}) < \epsilon\\
	1 & if & dist(p,\mathcal{C}) \geq \epsilon. 
\end{array}
\right.
\]
We denote de set of points of $M$ such that $n \in \mathbb{N}$ is a $(\sigma,\epsilon)$-hyperbolic time by $H_{n}(\sigma,\epsilon,f)$.
\end{definition}

\begin{proposition}
(Positive frequence). Given $\lambda > 0$ there exist $\theta > 0$ and $\epsilon_{0} > 0$ such that, for every $x \in M$ and $\epsilon \in (0,\epsilon_{0}]$,
\[
\#\{1 \leq j \leq n mid \,\, x \in H_{j}(e^{-\lambda \slash 4}, \epsilon, f) \} \geq \theta n,
\]
whenever $\frac{1}{n}\sum_{i=0}^{n-1}\log\|(Df(f^{i}(x)))^{-1}\|^{-1} \geq \lambda$ and $\frac{1}{n}\sum_{i=0}^{n-1}-\log dist_{\epsilon}(x, \mathcal{C}) \leq \frac{\lambda}{16 \beta}$.
\end{proposition}

Denote by $\mathcal{H}$ the set of point $x \in M$ such that
\[
\displaystyle \limsup_{n \to \infty} \frac{1}{n} \sum_{i=0}^{n-1} \log \parallel Df(f^{i}(p))^{-1}\parallel^{-1} < -\lambda.
\] 
and
\[
\displaystyle \limsup_{n \to \infty} \frac{1}{n} \sum_{i=0}^{n-1} - \log \text{dist}_{\delta}(p,\mathcal{C}) < \epsilon.
\]
If $f$ is non-uniformly expanding, it follows from the proposition that the points of $\mathcal{H}$ have infinitely many moments with positive frequency of hyperbolic times. In particular, they have infinitely many hyperbolic times.

The following proposition shows that the hyperbolic times are indeed zooming times, where the zooming contraction is $\alpha_{k}(r) = \sigma^{k/2}r$.

\begin{proposition}
Given $\sigma \in (0,1)$ and $\epsilon > 0$, there is $\delta,\rho > 0$, depending only on $\sigma$ and $\epsilon$ and on the map $f$, such that if $x \in H_{n}(\sigma,\epsilon,f)$ then there exists a neighbourhood $V_{n}(x)$ of $x$ with the following properties:

\begin{enumerate}
  \item[(1)] $f^{n}$ maps $\overline{V_{n}(x)}$ diffeomorphically onto the ball $\overline{B_{\delta}(f^{n}(x))}$;
  \item[(2)] $dist(f^{n-j}(y),f^{n-j}(z)) \leq \sigma^{j\slash 2} dist(f^{n}(y), f^{n}(z)), \text{for all} \, \, y,z \in V_ {n}(x)$ and $1 \leq j < n$.
  \item[(3)]$\log \frac{\mid \det Df^{n}(y)\mid}{\mid \det Df^{n}(z)\mid} \leq \rho d(f^{n}(y),f^{n}(z))$.
\end{enumerate}

for all $y,z \in V_{n}(x)$.
\end{proposition}

The sets $V_{n}(x)$ are called hyperbolic pre-balls and their images $f^{n}(V_{n}(x)) = B_{\delta}(f^{n}(x))$, hyperbolic balls.

\bigskip

In the following, we give definitions for a map on a metric space to have similar behaviour to maps with hyperbolic times and which can be found in \cite{Pi1}.  

Given $M$ a metric spaces and $f: M \to M$, we define for $p \in M$:
\[
\displaystyle \mathbb{D}^{-}(p) = \liminf_{x \to p} \frac{d(f(x),f(p)}{d(x,p)}
\]
Define also,
\[
\displaystyle \mathbb{D}^{+}(p) = \limsup_{x \to p} \frac{d(f(x),f(p)}{d(x,p)}
\]
We will consider points $x \in M$ such that 
\[
\displaystyle \limsup_{n \to \infty} \frac{1}{n} \sum_{i=0}^{n-1} \log \mathbb{D}^{-} \circ f^{i}(x) > 0.
\]  
The critical set $\mathcal{C}$ is the set of points $x \in M$ such that $\mathbb{D}^{-}(x) = 0$ or $\mathbb{D}^{+}(x) = \infty$. For  the non-degenerateness we ask that $\mathcal{C} \neq M$ and there exist $B, \beta >0$ such that

\begin{itemize}
 
  \item $\frac{1}{B} d(x,\mathcal{C})^{\beta} \leq \mathbb{D}^{-}(x) \leq \mathbb{D}^{+}(x) \leq B d(x,\mathcal{C})^{-\beta}, x \not \in \mathcal{C}$.
  
  For every $x, y \in M \backslash \mathcal{C}$ with $d(x,y) < d(x,\mathcal{C})/2$ we have
  
  \item $\mid \log \mathbb{D}^{-}(x) - \log \mathbb{D}^{-}(y) \mid \leq \frac{B}{d(x,\mathcal{C})^{\beta}} d(x,y)$.

\end{itemize}
 
With these conditions we can see that all the consequences for hyperbolic times are valid here and the expanding sets and measures are zooming sets and measures.

\begin{definition}
We say that a map is \emph{conformal at p} if $\mathbb{D}^{-}(p) = \mathbb{D}^{+}(p)$. So, we define
\[
\displaystyle \mathbb{D}(p) = \lim_{x \to p} \frac{d(f(x),f(p)}{d(x,p)}.
\]
\end{definition}

Now, we give an example of such an open non-uniformly expanding map.  

\subsection{Expanding sets on a metric space} Let $\sigma : \Sigma_{2}^{+} \to \Sigma_{2}^{+}$ be the one-sided shift, with the usual metric:
\[
\displaystyle d(x,y) = \sum_{n=1}^{\infty} \frac{\mid x_{n} - y_{n} \mid}{2^{n}},
\]
where $x = \{x_{n}\}, y = \{y_{n}\}$. We have that $\sigma$ is a conformal map such that $\mathbb{D}^{-}(x) = 2, \text{for all} \, \, \, x \in \Sigma_{2}^{+}$. Also, every forward invariant set (in particular the whole $\Sigma_{2}^{+}$)  and all invariant measures for the shift $\sigma$ are expanding (then they are zooming). In particular, if we consider an invariant set that is not dense such that the reference measure has a Jacobian with bounded distortion, we can obtain an open shift map with $H \neq \emptyset$. To be precise, by taking any (previously fixed) zooming set $\Lambda \subset \Sigma_{2}^{+}$ which is not dense such that the reference measure has a Jacobian with bounded distortion, we apply our Theorem \ref{A} to obtain an open zooming system and a Markov structure adapted to a hole $H \subset \Sigma_{2}^{+}$ such that $H \cap \Lambda = \emptyset$. It is enough to take $r_{0}>0$ as in Theorem \ref{A} such that one of the balls of the open cover is disjoint from $\Lambda$. Hence, we can apply our Theorems \ref{B} and \ref{C} to obtain equilibrium states. Afterwards, we obtain uniqueness as explained in Remark \ref{remark} and Section \ref{Uniqueness}.

\subsection{Zooming sets on a metric space (not expanding)} Let $\sigma : \Sigma_{2}^{+} \to \Sigma_{2}^{+}$ be the one-sided shift, with the following metric for $\sum_{n=1}^{\infty} b_{n} < \infty$:
\[
\displaystyle d(x,y) = \sum_{n=1}^{\infty} b_{n}\mid x_{n} - y_{n} \mid,
\]
where $x = \{x_{n}\}, y = \{y_{n}\}$ and $b_{n+k} \leq b_{n}b_{k}$ for all $n,k \geq 1$. By induction, it means that $b_{n} \leq b_{1}^{n}$. Let us suppose that $b_{n} \leq a_{n}:=(n+b)^{-a}, a>1, b>0$ for all $n \geq 1$. 

We claim that $a_{n}$ defines a Lipschitz contraction for the shift map. We require that there exists $n_{0} > 1$ such that $b_{n} > a_{1}^{n} \geq b_{1}^{n}$ for $n \leq n_{0}$. So, the contraction is not exponential. In fact, if $x,y$ belongs to the cylinder $C_{k}$ we have
\begin{eqnarray*}
	\displaystyle d(x,y) &=& \sum_{n=1}^{\infty} b_{n}\mid x_{n} - y_{n} \mid = \sum_{n=k+1}^{\infty} b_{n}\mid x_{n} - y_{n} \mid = \sum_{n=1}^{\infty} b_{n+k}\mid x_{n+k} - y_{n+k} \mid\\
	&\leq& b_{k} \sum_{n=1}^{\infty} b_{n}\mid x_{n+k} - y_{n+k} \mid = b_{k} d(\sigma^{k}(x),\sigma^{k}(y)) \leq a_{k} d(\sigma^{k}(x),\sigma^{k}(y)).
\end{eqnarray*}
It implies that
\begin{eqnarray*}
	\displaystyle d(\sigma^{i}(x),\sigma^{i}(y)) \leq  a_{k-i} d(\sigma^{k-i}(\sigma^{i}(x)),\sigma^{k-i}(\sigma^{i}(y)))= a_{k-i}d(\sigma^{k}(x),\sigma^{k}(y)), i \leq k.
\end{eqnarray*}
It means that the sequence $a_{n}$ defines a Lipschitz contraction, as we claimed.

Now, every forward invariant set (in particular the whole $\Sigma_{2}^{+}$)  and all invariant measures for the shift $\sigma$ are not expanding but they are  zooming. In particular, if we consider an invariant set that is not dense such that the reference measure has a Jacobian with bounded distortion, we can obtain an open shift map with $H \neq \emptyset$. To be precise, by taking any (previously fixed) zooming set $\Lambda \subset \Sigma_{2}^{+}$ which is not dense such that the reference measure has a Jacobian with bounded distortion, we apply our Theorem \ref{A} to obtain an open zooming system and a Markov structure adapted to a hole $H \subset \Sigma_{2}^{+}$ such that $H \cap \Lambda = \emptyset$. It is enough to take $r_{0}>0$ as in Theorem \ref{A} such that one of the balls of the open cover is disjoint from $\Lambda$. Hence, we can apply our Theorems \ref{B} and \ref{C} to obtain equilibrium states. Afterwards, we obtain uniqueness as explained in Remark \ref{remark} and Section \ref{Uniqueness}.

\subsection{Uniformly expanding maps} As can be seen in \cite{OV} Chapter 11, we have the so-called \textbf{\textit{uniformly expanding maps}} which is defined on a compact differentiable manifold $M$ as a $C^{1}$ map $f:M \to M$ (with no critical set) for which there exists $\sigma > 1$ such that
\[
\|Df(x)v\|\geq \sigma \|v\|, \,\, \text{for every} \,\, x \in M, v \in T_{x}M.
\]

For compact metric spaces $(M,d)$ we define it as a continuous map $f:M \to M$, for which there exists $\sigma > 1, \delta>0$ such that for every $x \in M$ we have that the image of the ball $B(x,\delta)$ contains a neighbourhood of the ball $B(f(x),\delta)$ and
\[
d(f(a),f(b)) \geq \sigma d(a,b), \,\, \text{for every} \,\, a,b \in B(x,\delta).
\]
We observe that the uniformly expanding maps on differentiable manifolds satisfy the conditions for the definition on compact metric spaces, when they are seen as Riemannian manifolds.

\subsection{Local diffeomorphisms}\label{local} As can be seen in details in \cite{A}, we will briefly describe a class of non-uniformly expanding maps.

Here we present a robust ($C^{1}$ open) classes of local diffeomorphisms (with no critical set) that are non-uniformly expanding. Such classes of maps can be obtained, e.g., through deformation of a uniformly expanding map by isotopy inside some small region. In general, these maps are not uniformly expanding: deformation can be made in such way that the new map has periodic saddles.

Let $M$ be a compact manifold supporting some uniformly expanding map $f_{0}$. $M$ could be the $d$-dimensional torus $\mathbb{T}^{d}$, for instance. Let $V \subset M$ be some small compact domain, so that the restriction of $f_{0}$ to $V$ is injective. Let $f$ be any map in a sufficiently small $C^{1}$-neighbourhood $\mathcal{N}$ of $f_{0}$ so that:

\begin{itemize}
	\item $f$ is \textit{volume expanding everywhere}: there exists $\sigma_{1} > 1$ such that
	\[
	|\det Df(x)| > \sigma_{1} \,\, \text{for every} \,\, x \in M;
	\]	
	
	\item $f$ is \textit{expanding outside} $V$: there exists $\sigma_{0} > 1$ such that
	\[
	\|Df(x)^{-1}\| < \sigma_{0} \,\, \text{for every} \,\, x \in M \backslash V;
	\] 
	
	\item $f$ is \textit{not too contracting on} $V$: there is some small $\delta > 0$ such that
	\[
	\|Df(x)^{-1}\| < 1 + \delta \,\, \text{for every} \,\, x \in V.
	\]
\end{itemize}

In \cite{A} it is shown that this class satisfy the condition for non-uniform expansion. In the following we show a Lemma from \cite{A} which proves that such maps are non-uniformly expanding with Lebesgue as a reference measure.
\begin{lemma}
	Let $B_{1},\dots, B_{k},B_{k+1}= V$ a partition of $M$ into domains such that $f$ is injective on $B_{j}, 1 \leq j \leq p+1$. There exists $\theta > 0$  such that the orbit of Lebesgue almost every point $x \in M$ spends a fraction $\theta$ of the time in $B_{1} \cap \dots B_{p}$, that is,
	\[
	\#\{0\leq j < n \mid f^{j}(x) \in B_{1} \cap \dots B_{p}\} \geq \theta n,
	\]
	for every large $n \in \mathbb{N}$.
\end{lemma}

\subsection{Open zooming systems from local diffeomorphisms} We can obtain an open zooming system such that the zooming set $\Lambda$ is disjoint from the hole $H$ ($\Lambda \cap H = \emptyset$) using a local diffeomorphism $f : M \to M$ which is non-uniformly expanding as in the subsection \ref{local}. 

Let the zooming set $\Lambda = \cap_{j=-\infty}^{\infty}f^{j}(M\backslash V)$ with positive Lebesgue measure $m$. Since $\Lambda \cap V = \emptyset$, we can take a zooming reference measure $\mu = m / m(\Lambda)$ which has a Jacobian with bounded distortion. The zooming set $\Lambda$ is disjoint from $V$ and we can take  the hole $H \subset V$ given by Theorem \ref{A} (and $\Lambda \cap H = \emptyset$). This setup now allows us to apply Theorems \ref{B} and \ref{C} to obtain existence and finiteness of (open) equilibrium state (uniqueness afterwards). We observe that in the work \cite{ALP2}[Lemma 2.1](3) we have that the Lebesgue measure has a Jacobian with bounded distortion.

As a concrete example on the interval $[a,b]$, we can take a dynamically defined Cantor set with positive Lebesgue measure. It can be seen in \cite{PT}[Chapter 4] as an expanding map $g$ over a disjoint union of intervals $P = I_{1} \uplus \dots \uplus I_{k}$ onto the interval $[a,b]$, where $I_{j} \in [a,b]$ is a compact interval for every $1 \leq j \leq k$, that is, every interval $I_{j}$ is taken onto $[a,b]$. Outside the union $P$ we can define the map to have a measurable map $f:[a,b] \to [a,b]$. In \cite{PT} we see that the map $g$ has bounded distortion and we can extend it to the map $f$ preserving this property. The hole $H$ can be taken outside the union $P$.

\end{document}